\documentclass[onefignum,onetabnum]{siamart190516}

\usepackage{psfrag,mathrsfs,color,empheq}
\usepackage{color,bm} 
\usepackage{epsfig,soul}
\usepackage{multirow,makecell}

\usepackage{mathtools} 
\usepackage{amssymb,bbm}

\newcommand{\tn}[1]{\quad \textnormal{#1}\quad }

\newcommand{\IR}{{\mathbb{R}}}

\newcommand{\CD}{{\mathfrak{D}}}

\newcommand{\bmt}{\left[ \begin{array}{ccccccccc}}
	\newcommand{\emt}{\end{array}\right]}
\newcommand{\bean}{\begin{eqnarray*}}
	\newcommand{\eean}{\end{eqnarray*}}
\newcommand{\bea}{\begin{eqnarray}}
	\newcommand{\eea}{\end{eqnarray}}
\newcommand{\eq}{\begin{equation}\begin{array}{lllllllll}}
		\newcommand{\ee}{\end{array}\end{equation}}
\newcommand{\eqn}{\begin{equation*}\begin{array}{lllllllll}}
		\newcommand{\een}{\end{array}\end{equation*}}

\graphicspath{{./Figures/}}

\usepackage{lipsum}
\usepackage{amsfonts}
\usepackage{graphicx}
\usepackage{epstopdf}
\usepackage{algorithmic}
\ifpdf
\DeclareGraphicsExtensions{.eps,.pdf,.png,.jpg}
\else
\DeclareGraphicsExtensions{.eps}
\fi

\newsiamremark{remark}{Remark}
\newsiamremark{hypothesis}{Hypothesis}
\crefname{hypothesis}{Hypothesis}{Hypotheses}
\newsiamthm{claim}{Claim}

\headers{Multigrid for elliptic optimal control}{Y. He and J. Liu}

\title{Smoothing analysis of two robust  multigrid methods  for  elliptic optimal control problems\thanks{Submitted to the editors DATE.
	}}

\author{Yunhui He\thanks{Department of Computer Science, The University of British Columbia, Vancouver, BC, V6T 1Z4, Canada.
		(\email{yunhui.he@ubc.ca}).}
	\and Jun Liu\thanks{Department of Mathematics and Statistics, Southern Illinois University Edwardsville, Edwardsville, IL 62026, USA.
		(\email{juliu@siue.edu}).}}

\usepackage{amsopn}

\begin{document}

\maketitle

\begin{abstract}
  In this paper we study and compare two multigrid relaxation schemes with coarsening by two, three, and four for solving elliptic  sparse optimal control problems with control constraints. First, we perform a detailed local Fourier analysis (LFA) of a well-known collective Jacobi relaxation (CJR) scheme, where the optimal smoothing factors are derived. This insightful analysis reveals that the optimal relaxation parameters depend on  mesh size and regularization parameters, which was not investigated in literature. Second, we propose and analyze a new mass-based Braess-Sarazin relaxation (BSR) scheme, which is proven to provide smaller smoothing factors than the CJR scheme when $\alpha\ge ch^4$ for a small constant $c$. Here $\alpha$ is the regularization parameter and $h$ is the spatial mesh step size. These schemes are successfully extended to control-constrained cases through the semi-smooth Newton method. Coarsening by three or four with BSR is competitive with coarsening by two. 
  Numerical examples are presented to validate our theoretical outcomes. The proposed inexact BSR (IBSR) scheme, where two preconditioned conjugate gradients iterations are applied to the Schur complement system, yields better computational efficiency than the CJR scheme.
\end{abstract}


 
\section{Introduction}

Optimal control problems with partial differential equation (PDE) constraints \cite{Lions1971,Hinze2009,Troltzsch2010} appear ubiquitously in all disciplines of science and engineering that involving PDE models. In the past few decades, many efficient numerical algorithms \cite{Ito2008,Ulbrich2011,Borzi2012} were developed for solving such PDE-constrained optimization problems, which are usually much more computationally expensive than solving standalone PDE  models due to extra constraints and non-smooth cost functionals. 
In this paper, we will study and compare two parameter-robust geometric multigrid methods for solving the linear elliptic optimal control problem with control constraints
and $L^1$ cost functional that promoting sparsity in control design \cite{Stadler2007}. 

Let $\Omega\subset\IR^2$ be a bounded and open domain with Lipschitz boundary $\partial\Omega$.
We consider the following  distributed optimal control problem \cite{Stadler2007,Porcelli2017} of
\eq \label{goal}
\min_{u\in U_{ad}}\quad J(y,u)=\frac{1}{2} \|y-g\|^2_{L^2(\Omega)}+\frac{\alpha}{2}\|u\|^2_{L^2(\Omega)}
+\beta\|u\|_{L^1(\Omega)}
\ee
subject to a Poisson equation
\eq \label{state}
-\Delta y =f+u\ \tn{in} \Omega \quad \tn{and}\quad y=0\ \tn{on}\ \partial\Omega,
\ee
where  $u\in U_{ad}$ is the control, $g\in L^2(\Omega)$ is the target state, 
$\alpha>0$ and $\beta\ge 0$ are regularization parameters, and $f\in L^2(\Omega)$. 
Here the boxed convex set $U_{ad}$ of admissible controls is defined as
\eqn \label{adcontrol}
U_{ad}=\{u\in L^2(\Omega)\ |\ u_0\le u \le u_1 \quad a.e.\tn{in}\Omega\},
\een
where $u_0,u_1\in L^{2} (\Omega)$ with $u_0< 0< u_1$.
It is desirable to develop efficient numerical algorithms with mesh-independent convergence rates that are also robust with respect to the parameters $\alpha$ and $\beta$.
We will focus on multigrid methods due to the ellipticity.

Broadly speaking, they are two different groups of numerical algorithms for solving such constrained optimization problem (\ref{goal}-\ref{state}): {discretize-then-optimize (DTO)} and {optimize-then-discretize (OTD)}. 
In  DTO approaches, one first discretizes the continuous optimization problem (\ref{goal}-\ref{state}) to 
obtain a  large-scale discretized finite-dimensional optimization problem, 
which is then approximately solved by black-box numerical optimization algorithms (e.g. gradient-based method \cite{Schindele2016} and interior-point method \cite{Pearson2019,Bergamaschi2021}) that are scalable to the number of decision variables.  
Such a DTO approach is attractive in some applications since it provides more flexibility in handling additional control and/or state constraints, gradient bounds, and more advanced regularization terms.
In OTD approaches, one first derives the first-order necessary optimality PDE system, and then
discretizes the continuous PDE optimality system with appropriate discretization schemes, such as  finite difference and finite element method.
The resultant discretized linear/nonlinear systems can be then solved by various efficient iterative solvers.
Very recently, some hybrid optimize-discretize-optimize approaches based on inexact alternating direction method of multipliers (ADMM) algorithms \cite{Song2017,Song2018a,Song2018,Chen2020} were also proposed for such PDE-constrained optimization problems.
In both approaches,  some efficient and effective structure-exploiting preconditioning techniques (see e.g. \cite{Rees2010,Bai2010,Herzog2010,Schiela2014,Gong2017}) are often required to achieve mesh-independent and parameter-robust  fast convergence rates.

In this paper, we will follow the OTD approach within the framework of semi-smooth Newton (SSN) method \cite{Chen2000,Michael2004,Hintermuller2002}, which was shown to be equivalent to the  primal-dual active-set method \cite{Bergounioux1999}.
More specifically, we will focus on developing and analyzing two multigrid algorithms for solving the Jacobian linear system in each SSN iteration, where the special saddle-point structure of Jacobian is utilized to design effective relaxation schemes. 
For elliptic optimal control problems without sparsity cost term, both linear and nonlinear multigrid methods were extensively studied with convergence analysis under certain assumptions, see for example
\cite{Borz2005,Borzi2009,Lass2009,takacs2011convergence,Schberl2011,Borzi2012,Takacs_2013,Liu2014,Drgnescu2016} and the references therein. 
However, to the best of our knowledge, the precise convergence rates of these developed multigrid algorithms for optimality PDE system were rarely estimated, which is significantly more difficult than the corresponding multigrid algorithm for a single PDE.
We contribute to fill this gap by conducting a through  Local Fourier analysis (LFA) of the studied two relaxation schemes: collective Jacobi relaxation (CJR) and mass-based Braess-Sarazin  relaxation (BSR).

LFA  is a useful tool to help better understand and design fast multigrid methods for parameter-dependent optimal control problems,   where the regularization parameters often post some challenges in achieving fast and robust convergence rates. 
It has been successfully applied to multigrid methods for optimal control problems.  In \cite{borzi2002accuracy}, convergence factor estimates of the two-grid method with  collective Gauss–Seidel relaxation for the unconstrained optimality system were obtained by LFA.  
The similar idea was extended to nonlinear full approximation storage (FAS) multigrid for handling constrained nonlinear optimality system \cite{Borz2005,Lass2009,Borzi2012}.
In \cite{takacs2011convergence}, the authors showed the mesh-independent convergence rates of collective
Jacobi and Gauss-Seidel relaxations for unconstrained optimal control problem, but they did not derive optimal damping parameter.  In \cite{Schberl2011}, a robust multigrid algorithm based on a transformed  symmetric Uzawa-type smoother was analyzed, where the deterioration of the  convergence rates is observed if $\alpha$ gets very small compared to  $h^4$.
In \cite{engel2011multigrid}, the authors developed a
multigrid method with a constraint preconditioned Richardson iteration as the smoother
in the framework of primal-dual active-set method, which requires a similar condition $\alpha>h^4/4$ to assure convergence.   In \cite{pillwein2011computing},  LFA based on symbolic computation (on Mathematica) was used to estimate the smoothing factor of collective Jacobi relaxation based on finite element method for unconstrained problem. In \cite{he2021novel}, the author studies a novel Braess-Sarazin relaxation scheme for the unconstrained cases discretized by finite element method, where the inverse of mass matrix is approximated by the Laplacian discretized by the five-point finite difference stencil.

We point out that most of the existing LFA studies on optimal control problems only focus on standard coarsening and only \textit{numerically} compute the optimal damping parameter or smoothing factor. In this work, we derive optimal damping parameter and optimal smoothing factor for the well-known CJR with coarsening by two, three and four, without any restrictions on $\alpha$ and $h$.   From this analysis, we see that the optimal damping parameter is dependent on mesh size and regularization parameter, then in our multigrid methods, we use such optimal damping parameter to get improved performance. Moreover, our  analysis  clearly tells us how the optimal smoothing factor is changed as a function of mesh size and regularization parameter.  As an improvement, we further propose a mass-based BSR scheme. We derive a upper bound on the optimal smoothing factor for BSR,  which shows that BSR is  unconditionally convergent,  and that  the optimal smoothing factor is smaller than the CJR scheme. The BSR is a  parameter-robust
multigrid algorithm. Our numerical results show that BSR   outperforms CJR.  Although our analysis only focuses  on unconstrained optimality system, we numerically extend our algorithms to the Jacobian systems from control-constrained cases within the framework of SSN method. 
Finally, we mention that multigrid algorithms based on  coarsening by three or four were rarely studied in literature; see \cite{Lass2009,JEDendy2010,Yavneh2012,liu2020multigrid}.
One possible reason is significantly degraded convergence rates compared to  standard coarsening.
In particular, this work is inspired by the recent work \cite{he2022optimal} on multigrid solvers with coarsening by three for the Stokes system, where the author proves a mass-based BSR scheme  achieves a convergence rate very close to standard coarsening.
In this work, the similar nice conclusions are obtained for our considered elliptic optimal control problems.

The paper is organized as follows. In the next section we present the first-order optimality system and the corresponding SSN iteration.
In Section \ref{SecLFA}, LFA for both CJR and BSR schemes are provided in detail to  estimate the optimal smoothing factor and derive the corresponding optimal damping parameter.
In Section \ref{SecNum},  we present some numerical examples (including an inexact BSR scheme)   to verify our theoretical findings.
Finally, some conclusion and remarks are given in Section \ref{SecFinal}.
The practical implementation of both CJR and BSR schemes is given in Appendix A.
\section{First-order optimality system and semi-smooth Newton method} \label{SecSSN}
Following \cite{Stadler2007}, the first-order necessary optimality conditions of (\ref{goal}--\ref{state}) are given by 
\eq \label{opt1}
-\Delta y -u=f\ \tn{in} \Omega \quad \tn{and}\quad y=0\ \tn{on}\ \partial\Omega, \\
-\Delta p +y=g\ \tn{in} \Omega \quad \tn{and}\quad p=0\ \tn{on}\ \partial\Omega, 
\ee
together with the complementary condition (in explicit form)
\eq
u=\Phi_{\alpha,\beta}(p),
\ee
where $p$ is the adjoint state and 
\eq \label{uproj}
\Phi_{\alpha,\beta}(p)&:=&\frac{1}{\alpha}[\max(0,p-\beta)+\min(0,p+\beta)  \\  &&\quad -\max(0,p-\beta-\alpha u_1)-\min(0,p+\beta-\alpha u_0) ] .
\ee
For our considered strongly convex problems, such necessary optimality conditions are also sufficient.
Figure \ref{FigPhiShape} illustrates how the control $u$ depends on the adjoint state $p$ via the nonlinear mapping $u=\Phi_{\alpha,\beta}(p)$ with a set of selected parameters, where the sparsity is due to $u\equiv 0$ whenever $|p|\le\beta$. In particular,  if $\beta$ is sufficiently large, then the optimal control will be just zero everywhere.
\begin{figure}[htp!]
	\begin{center}
		\includegraphics[width=1\textwidth]{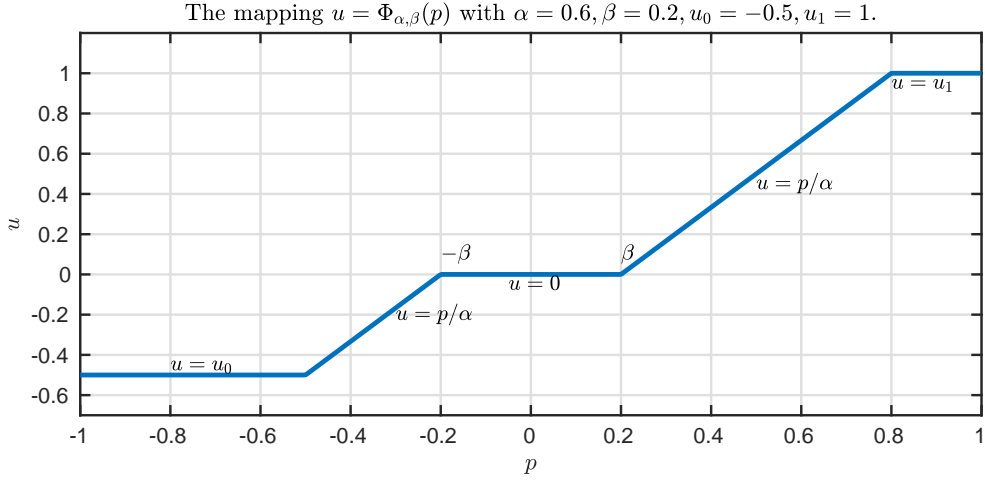}
	\end{center} 
	\caption{The piecewise shape of the  function or mapping $u=\Phi_{\alpha,\beta}(p)$.}
	\label{FigPhiShape}
\end{figure}

{By} substituting (\ref{uproj}) into {the} optimality conditions (\ref{opt1}) to eliminate $u$, we {thus} obtain 
the following reduced non-smooth nonlinear optimality system  
\eq \label{opt1A}
-\Delta y -\Phi_{\alpha,\beta} (p)&=f\ \tn{in} \Omega \quad \tn{and}\quad y=0\ \tn{on}\ \partial\Omega, \\
-\Delta p +y&=g\ \tn{in} \Omega \quad \tn{and}\quad p=0\ \tn{on}\ \partial\Omega. 
\ee
If the sparsity $L^1$ cost term is not present (that is $\beta=0$), then $\Phi_{\alpha,\beta}$ is simplified into the following pointwise projection (onto $U_{ad}$) characterization
\eqn \label{uproj0}
 \Phi_{\alpha,\beta} (p)=\frac{1}{\alpha}\left[p-\max(0,p-\alpha u_1)-\min(0,p-\alpha u_0) \right]=\min\left(u_0,\max\left(u_1,p/\alpha\right)\right) .
\een
If $\beta>0$ without control constraints (i.e. $U_{ad}=L^2(\Omega)$), then $\Phi_{\alpha,\beta}$ becomes
\eqn \label{uproj_beta}
 \Phi_{\alpha,\beta} (p)=\frac{1}{\alpha}\left[\max(0,p-\beta)+\min(0,p+\beta) \right] .
\een
If $\beta=0$ without control constraints, then
it gives the linear relation
$\Phi_{\alpha,\beta} (p)=p/\alpha$.

\subsection{The linear case with $\beta=0$ and no control constraints} 

With $\beta=0$ and no control constraints ($U_{ad}=L^2(\Omega)$), the optimality system (\ref{opt1A})   simplifies to  
\eq \label{opt1U}
-\Delta y -p/\alpha&=f\ \tn{in} \Omega \quad \tn{and}\quad y=0\ \tn{on}\ \partial\Omega, \\
-\Delta p +y&=g\ \tn{in} \Omega \quad \tn{and}\quad p=0\ \tn{on}\ \partial\Omega. 
\ee
Let $L_h=-\Delta_h$ denotes the discretized negative Laplacian by the five-point center finite difference method (with a uniform mesh step size $h=1/N$). The full finite difference discretization of system (\ref{opt1U}) leads to a two-by-two block linear system 
\eq \label{opt1Uh}
A_h v_h:=\bmt L_h & -I_h/\alpha \\ I_h & L_h \emt 
\bmt y_h\\ p_h \emt=  \bmt f_h\\ g_h \emt =:b_h,
\ee
where $I_h\in\IR^{n\times n}$ with $n=(N-1)^2$ is an identity matrix and the vectors $f_h,g_h,y_h,p_h$ denote the corresponding functions $f,g,y,p$ over all the spatial grid points $\Omega_h=\{(ih,jh):1\le i,j\le (N-1)\}$.
If the $P_1$ finite element method is used, then the matrix $L_h$ and $I_h$ will be replaced
by the corresponding stiffness and mass matrix, respectively.  
The proposed algorithms can be extended to treat finite element discretization after necessary modifications; see \cite{he2021novel} for related discussion.
\subsection{The nonlinear case with $\beta>0$ and  control constraints}  
To solve the non-smooth nonlinear optimality system (\ref{opt1A}), we will use the well-established  SSN method for handling the non-smooth operator $\Phi_{\alpha,\beta}$. 
Define the pointwise max- and min-operators: $F_{\max}(v)=\max(0,v)$ and $F_{\min}(v)=\min(0,v)$, respectively. The generalized derivative of $F_{\max}$ and $F_{\min}$ given by 
$\partial F_{\max}(v)(\bm x)=\mathbbm{1}_{\{v(\bm x)\ge 0\}}$ 
and $\partial F_{\min}(v)(\bm x)=\mathbbm{1}_{\{v(\bm x)\le 0\}}$, where $\mathbbm{1}_X$ is the indicator function on a set $X$.
With these notations, the generalized derivative of $\Phi_{\alpha,\beta}$ (see Figure \ref{FigPhiShape}) reads 
\eq \label{uproj_gd}
\partial\Phi_{\alpha,\beta}(p) =\frac{1}{\alpha} [&\partial F_{\max}(p-\beta)+\partial F_{\min}(p+\beta)\\
&-\partial F_{\max}(p-\beta-\alpha u_1)-\partial F_{\min}(0,p+\beta-\alpha u_0)  ].
\ee
We now formulate the full discretization of the nonlinear optimality system (\ref{opt1A}) into
\eq \label{opt1Ah}
\mathcal{F}(y_h,p_h):=
\bmt L_h y_h -\Phi_{\alpha,\beta} (p_h)-f_h\\
L_h p_h +y_h -g_h
\emt =
\bmt 0\\ 0
\emt ,
\ee
which can be solved by the SSN iterations (start with a given initial guess $(y_h^{(0)},p_h^{(0)})$):
\eq \label{SSNiter}
\bmt y_h^{(k+1)}\\p_h^{(k+1)} \emt=
\bmt y_h^{(k)}\\p_h^{(k)} \emt-   \left[\partial\mathcal{F}(y_h^{(k)},p_h^{(k)})\right]^{-1}   \mathcal{F} (y_h^{(k)},p_h^{(k)})
=:\bmt y_h^{(k)}\\p_h^{(k)} \emt- \bmt \delta y_h\\ \delta p_h \emt.
\ee
Clearly, at each SSN iteration we have to solve a Jacobian system
of structure:
\eq \label{opt1Jh}
\partial\mathcal{F}(y_h^{(k)},p_h^{(k)}) \bmt \delta y_h\\ \delta p_h \emt=\bmt L_h & -\partial\Phi_{\alpha,\beta}(p^{(k)}) \\ I_h & L_h \emt 
\bmt \delta y_h\\ \delta p_h \emt= \mathcal{F} (y_h^{(k)},p_h^{(k)}),
\ee
where $\partial\Phi_{\alpha,\beta}(p^{(k)})=\CD_h^{(k)}/\alpha$ with $\CD_h^{(k)}$  being a diagonal $\{0,1\}$ matrix (with only 0 or 1 on the main diagonal) depending on $p^{(k)}$. Notice the value of $\beta$ only changes the 0-1 pattern of $\CD_h^{(k)}$ by cutting off those points with $|p(\bm x)|\le \beta$.  
In practice for better efficiency, the Jacobian systems (\ref{opt1Jh}) are only approximately solved, say by a few multigrid iterations, which gives so-called inexact SSN method.
To achieve a robust global convergence, we will also incorporate the back-tracking line-search based globalization strategy as introduced in \cite{Martnez1995} and also used in \cite{Porcelli2017}.
Since our main focus lies in developing fast multigrid solver for solving the Jacobian systems, we will not further discuss the convergence issues of such an inexact SSN method.
Compared with the discretized unconstrained optimality system (\ref{opt1Uh}), the only difference in the Jacobian system (\ref{opt1Jh}) is $I_h$ becomes $\CD_h^{(k)}$ in the (1,2) block, for which our proposed multigrid solvers are expected to work very effectively since their fast convergence rates are very robust with respect to the parameter $\alpha$.

\section{LFA of collective Jacobi relaxation and Braess-Sarazin relaxation} \label{SecLFA}
We consider multigrid  methods for solving linear system (\ref{opt1Uh}). In multigrid, the fixed-point type relaxation scheme   has the following form
\begin{equation}\label{eq:relaxation-scheme}
	v_h^{k+1}=v_h^{k}+\omega B^{-1}_h(b_h-A_hv_h^k),
\end{equation}
where $B_h$ approximates $A_h$ and $\omega\in\IR$ is a damping parameter to be determined.

In the  CJR, the matrix $B_h$ in \eqref{opt1Uh} is given by
\begin{equation}\label{eq:CJ-relaxation}
	B_h=B_{J}:= \begin{bmatrix}
		D_h& -  I_h/\alpha\\
		I_h & D_h
	\end{bmatrix}, \quad\text{with}\,\,  D_h={\rm diag}(L_h).
\end{equation}  
To obtain a better relaxation scheme, it is important to construct a good approximation to the negative Laplacian matrix $L_h$. In \cite{CH2021addVanka}, it has been shown that the mass matrix $Q_h$ (obtained from bilinear finite elements) with stencil representation
\begin{equation}\label{eq:mass-stencil}
	Q_h =\frac{h^2}{36}\begin{bmatrix}
		1 & 4 & 1  \\
		4 & 16& 4 \\
		1& 4 & 1   
	\end{bmatrix}
\end{equation} 
is a good approximation to the inverse of $L_h$.   
Due to the saddle-point structure of $A_h$ in \eqref{opt1Uh}, here, we propose a mass-based  BSR scheme, where $B_h$ in \eqref{opt1Uh} reads
\begin{equation}\label{eq:exact-BSR-relaxation}
	B_h=B_m:= \begin{bmatrix}
		C_h& -  I_h/\alpha\\
		I_h & L_h
	\end{bmatrix}, \quad\text{with}\,\, C_h=Q^{-1}_h.
\end{equation} 
By (\ref{eq:relaxation-scheme}), the relaxation error operator  of both CJR  and BSR for solving \eqref{opt1Uh}  is
\begin{equation}\label{eq:relaxation-error-operator}
	S_h = I-\omega B^{-1}_h A_h.
\end{equation} 
Once $B_h$ is chosen, the next task is to select a good or  optimal damping parameter such that the multigrid converges as fast as possible. To achieve this goal, we apply LFA \cite{trottenberg2000multigrid,wienands2004practical} to identify  optimal damping parameter and examine smoothing property of CJR and exact BSR with coarsening by two, three and four. For practical implementation, we propose to use an inexact BSR, where we apply a few preconditioned conjugate gradients (PCG) iterations to solve the Schur complement system inexactly. We emphasize that the matrices $B_h^{-1}$ and $Q_h^{-1}$ are never explicitly constructed, since only the matrix-vector product $B_h^{-1}z_h$ is needed in relaxation iterations.
See Appendix A for a detailed discussion.
\subsection{Local Fourier analysis}
There are two important factors in LFA: LFA smoothing factor and LFA two-grid convergence factor, which are computed from the {\it symbol} of corresponding operators. In many cases, LFA smoothing factor can offer a good prediction of LFA two-grid convergence factor and actual multigrid performance. Thus,   we focus on analyzing the LFA smoothing factor.   For multigrid methods, we consider $q$-coarsening, where $q=2,3,4$. The corresponding low and high frequencies for LFA is defined as
\begin{equation}
	\boldsymbol{\theta}=(\theta_1,\theta_2) \in T^{\rm L_q} =\left(-\frac{\pi}{q}, \frac{\pi}{q}\right]^2, \quad 
	\boldsymbol{\theta} \in T^{\rm H_q} =\left(-\frac{\pi}{2}, \frac{3 \pi}{2}\right]^2 \backslash T^{\rm L_q}. 
\end{equation} 
We now  give some definitions of LFA, following the standard notations in \cite{trottenberg2000multigrid}. 
\begin{definition}
	Let $L_h$ be a scalar operator represented by stencil $[s_{\boldsymbol{\kappa}}]_h$ acting on grid $G_h$ as
	\begin{equation*}
		L_{h}\phi_{h}(\boldsymbol{x})=\sum_{\boldsymbol{\kappa}\in{V}}s_{\boldsymbol{\kappa}}\phi_{h}(\boldsymbol{x}+\boldsymbol{\kappa}h),
	\end{equation*}
	where  $s_{\boldsymbol{\kappa}}\in \mathbb{R}$ or $\mathbb{C}$  is constant,   $\phi_{h}(\boldsymbol{x}) \in l^{2} ({G}_{h})$, and  ${V}$ is a finite index set. 
	Then, the  symbol of $L_{h}$ is given by
	\begin{equation}\label{eq:symbol-form}
		\widetilde{L}_{h}(\boldsymbol{\theta})=\displaystyle\sum_{\boldsymbol{\kappa}\in{V}}s_{\boldsymbol{\kappa}}e^{i\boldsymbol{\theta}\cdot\boldsymbol{\kappa}},\,\, i^2=-1. 
	\end{equation} 
\end{definition} 
\begin{definition}
	The LFA smoothing factor for the relaxation error operator $S_h$ in \eqref{eq:relaxation-error-operator} is defined as
	\begin{equation}\label{eq:loc-mu-Sh}
		\mu_{\rm loc}(S_h) = \max_{\boldsymbol{\theta} \in T^{\rm H_q}}\{\rho(\widetilde{S}_h(\omega,\boldsymbol{\theta}))\},
	\end{equation} 
	where $\rho(\widetilde{S}_h(\omega,\boldsymbol{\theta}))$ stands for the spectral radius of matrix symbol $\widetilde{S}_h(\omega,\boldsymbol{\theta})$.
	Moreover, the LFA optimal smoothing factor for $S_h$  is given by
	\begin{equation}\label{eq:def-opti-mu}
		\mu_{\rm opt}(S_h) = \min_{\omega \in\mathbb{R}} \mu_{\rm loc}(S_h(\omega)).
	\end{equation} 
\end{definition}

We remark that $\widetilde{S}_h$ in \eqref{eq:loc-mu-Sh} is a $2\times 2$ matrix, since $M_h$ and $A_h$ are $2\times 2$ block matrices and the symbol of each block is a scalar.  We are interested in analytically finding the  optimal solution of \eqref{eq:def-opti-mu}:  identify the optimal algorithmic parameter $\omega$ and the corresponding optimal smoothing factor. 

The stencil presentation of $L_h$ is
\begin{equation}
	L_h = \frac{1}{h^2}\begin{bmatrix}
		& -1 &   \\
		-1 & 4& -1 \\
		& -1 &    
	\end{bmatrix}.
\end{equation} 
Using \eqref{eq:symbol-form},  the symbol of $L_h$ reads
\begin{equation}\label{eq:symbol-Lh}
	\widetilde{L}_h=\frac{1}{h^2}(4-2\cos \theta_1 -2\cos \theta_2)=:a.
\end{equation}
It follows that
\begin{equation}
	\widetilde{A}_h = \begin{bmatrix}
		a& -1/\alpha     \\
		1 & a
	\end{bmatrix}.
\end{equation}

\subsection{The CJR scheme} 

From \eqref{eq:CJ-relaxation}, the symbol of ${B}_{J}$ is 
\begin{equation}\label{eq:collective-Jacobi-symbol}
	\widetilde{B}_{J}= \begin{bmatrix}
		a_1& -1/\alpha\\
		1  & a_1
	\end{bmatrix},\quad \text{with}\,\, a_1=\frac{4}{h^2}.
\end{equation} 
To find the optimal smoothing factor \eqref{eq:def-opti-mu}  for the CJR scheme, we have to find the eigenvalues of $\widetilde{B}_{J}^{-1}\widetilde{A}_h$ . We first compute the determinant of $\widetilde{A}_h-\lambda \widetilde{B}_{J}$:
\begin{equation*}
	| \widetilde{A}_h-\lambda\widetilde{B}_{J}| =
	\begin{vmatrix}
		a - \lambda a_1      & -1/\alpha (1-\lambda)  \\
		1-\lambda     &   a-\lambda a_1
	\end{vmatrix} = (a - \lambda a_1 )^2+ 1/\alpha(1-\lambda)^2. 
\end{equation*}
Then, it can be shown that the two eigenvalues of   $\widetilde{B}_{J}^{-1}\widetilde{A}_h$  are 
\begin{equation}
	\lambda_{1,2}=\frac{\frac{a}{a_1}  \pm i \frac{1}{\sqrt{\alpha} a_1}}{1\pm i \frac{1}{\sqrt{\alpha} a_1}}.
\end{equation}
Let $\tau=\frac{a}{a_1}$ and $\gamma =\frac{1}{\sqrt{\alpha} a_1}$. Then,
\begin{equation*}
	|\lambda(I-\omega \widetilde{B}_{J}^{-1}\widetilde{A}_h)|=|1-\omega \lambda_{1,2}|
	= \left |1-\omega \frac{\tau + i \gamma}{1+ i \gamma}\right|= \left| \frac{(1-\omega \tau)+(1-\omega)\gamma i}{1+i \gamma }\right|.
\end{equation*}

Since  $a_1$ is a constant independent of $\boldsymbol{\theta}$, we now simplify  the optimal smoothing factor \eqref{eq:def-opti-mu} for the CJR with $q=2,3,4$.
\begin{theorem}\label{thm:CJ-simplify-mu}
	For the CJR scheme, let  $\tau_1, \tau_2$ be the minimum and maximum of $\tau=\frac{a}{a_1}$ with $\boldsymbol{\theta} \in T^{\rm H_q}$, respectively, and $\tau_0=\frac{2}{\tau_1+\tau_2}$.  Assume that $0<\tau_1<1<\tau_2$ and $\tau_0\leq 1$. Then 
	\begin{equation}\label{eq:general-minmax-form}
		\mu_{\rm opt,CJR}
		=\min_{\omega \in [\tau_0,\infty)} \left\{\frac{\sqrt{ |1-\tau_2 \omega |^2+|(1-\omega)\gamma|^2} }{|1+i \gamma| } \right \},
	\end{equation}
	where $\gamma = \frac{1}{\sqrt{\alpha} a_1}=\frac{h^2}{4 \sqrt{\alpha} }$. 
\end{theorem}
\begin{proof}
	\begin{align} 
		&\mu_{\rm loc}(S_h(\omega))=\max_{\boldsymbol{\theta}\in T^{\rm H_q}}|\lambda(I-\omega \widetilde{B}_{J}^{-1}\widetilde{A}_h)|\nonumber\\
		&=\max \left\{\zeta_1=\frac{\sqrt{|1-\tau_2 \omega |^2+|(1-\omega)\gamma|^2} }{|1+i \gamma| },\,\, \zeta_2=\frac{\sqrt{|1-\tau_1\omega |^2+|(1-\omega)\gamma |^2} }{|1+i \gamma |}\right \}\nonumber.
	\end{align}  
	Note that when $\omega=\tau_0$,     $\zeta_1(\tau_0)=\zeta_2(\tau_0)$. If $\omega<\tau_0$, then $\zeta_2(\omega)>\zeta_2(\tau_0)$, and $\mu_{\rm loc }(S(\omega))>\mu_{\rm loc }(S(\tau_0))$. Thus, to minimize $\mu_{\rm loc }(S(\omega))$ over $\omega \in \mathbb{R}$, we only need to consider $\omega\in[\tau_0,\infty)$. 
	For $\omega \geq \tau_0$, $\zeta_1>\zeta_2$. Thus,  
	\begin{equation*}
		\mu_{\rm opt, CJR}=\min_{\omega \in\mathbb{R}} \mu_{\rm loc}(S_h(\omega))
		=\min_{\omega \in [\tau_0,\infty)} \left\{\frac{\sqrt{ |1-\tau_2 \omega |^2+|(1-\omega)\gamma|^2} }{|1+i \gamma| } \right \},
	\end{equation*}
	which completes the proof.
\end{proof}

From Theorem \ref{thm:CJ-simplify-mu}, it is important to estimate the range of $\tau =\frac{a}{a_1}$, i.e., $\tau_1, \tau_2$, to solve  \eqref{eq:general-minmax-form}.  Since $\tau = \widetilde{D}_h^{-1} \widetilde{L}_h$, we first  investigate  the optimal smoothing factor of weighted Jacobi relaxation scheme for the scalar Laplacian for $q=2, 3, 4$. The standard coarsening case with $q=2$ is well-known as the following lemma.
\begin{lemma}\label{lemma:Jacobi-H2-opt-mu} 
	For  standard  coarsening ($q=2$),   $\tau=\frac{a}{a_1} \in \left [\frac{1}{2}, 2\right]$ for $\boldsymbol{\theta} \in T^{\rm H_2}$.   Moreover, the optimal smoothing factor of weighted Jacobi relaxation scheme for the scalar Laplacian is  achieved at $\omega =\frac{4}{5}$, given by
	\begin{equation}
		\min_{\omega} \max_{\boldsymbol{\theta} \in T^{\rm H_2}} \{|\lambda(1-\omega \widetilde{D}_h^{-1} \widetilde{L}_h)|\} 
		=\min_{\omega} \max  \left\{|1-\omega /2|, |1-2\omega| \right\} =\frac{3}{5}.
	\end{equation} 
\end{lemma}
We now give the range of  $(\cos\theta_1,\cos\theta_2)$  with $\boldsymbol{\theta}\in T^{\rm H_q}$ and  $q=3, 4$,  which will be very useful for conducting our optimal smoothing analysis for $q=3, 4$. 

For $\boldsymbol{\theta}\in T^{\rm H_3}$,  it can be shown  that 
\begin{equation}\label{eq:H3-range}
	(\cos\theta_1,\cos\theta_2)  \in \mathcal{D}^{(3)}=[-1,1]\times [-1,1/2] \bigcup [-1, 1/2]\times [1/2,1].
\end{equation}

For $\boldsymbol{\theta}\in T^{\rm H_4}$,  it can be shown that 
\begin{equation}\label{eq:H4-range}
	(\cos\theta_1,\cos\theta_2) \in \mathcal{D}^{(4)}=[-1,\sqrt{2}/2]\times [\sqrt{2}/2, 1]\bigcup [-1,1] \times [-1,\sqrt{2}/2].  
\end{equation}

\begin{theorem}\label{thm:Jacobi-H3-opt-mu} 
	For   coarsening by three ($q=3$),  $\tau=\frac{a}{a_1} \in \left [\frac{1}{4}, 2\right]$ for $\boldsymbol{\theta} \in T^{\rm H_3}$.   Moreover, the optimal smoothing factor of weighted Jacobi relaxation scheme for the scalar Laplacian is given by
	\begin{equation}
		\min_{\omega} \max_{\boldsymbol{\theta} \in T^{\rm H_3}} \{|\lambda(1-\omega \widetilde{D}_h^{-1}\widetilde{L}_h)|\} 
		=\min_{\omega} \max  \left\{|1-\omega /4|, |1-2\omega| \right\} =\frac{7}{9}\approx 0.778,
	\end{equation} 
	provided that $\omega =\frac{8}{9}$.
\end{theorem}
\begin{proof}
	Let $\eta_1=\cos \theta_1$, $\eta_2=\cos \theta_2$.   Note that $\widetilde{D}_h^{-1}\widetilde{L}_h=\frac{1}{2}(2-\cos\theta_1-\cos\theta_2)$. Let $\psi(\eta_1,\eta_2)=\frac{1}{2}(2-\eta_1-\eta_2)$. If 
	$\psi'_{\eta_1}(\eta_1,\eta_2)=\psi'_{\eta_2}(\eta_1,\eta_2)=0$, we have $\eta_1=\eta_2=0$. Thus, $\psi(0,0)=1$ is a possible extreme value of $\psi(\eta_1,\eta_2)$.

	Next, we look for the extreme values of $\psi(\eta_1,\eta_2)$ at the boundary of $\mathcal{D}^{(3)}$, see \eqref{eq:H3-range}. 
	Due to the symmetry of $\psi(\eta_1,\eta_2)$, that is $\psi(\eta_1,\eta_2)=\psi(\eta_2,\eta_1)$, we only need to find the extreme values of $\psi(\eta_1,\eta_2) $ at $\partial \mathcal{D}_1, \partial \mathcal{D}_2$  and $\partial \mathcal{D}_2$, where
	\begin{equation*}
		\partial \mathcal{D}_1 =\{-1 \}\times [-1,1], \,\partial \mathcal{D}_2= \{1\}\times \left[-1,1/2\right], \, \partial  \mathcal{D}_3 =\left [1/2,1\right]\times \{1/2\}.
	\end{equation*} 
	
	\begin{itemize}
		
		\item For $ (\eta_1,\eta_2)\in \partial \mathcal{D}_1$, $\psi(\eta_1,\eta_2) =\psi(-1,\eta_2) =\frac{1}{2}(3-\eta_2).$  Since  $\eta_2 \in[-1,1]$, the extreme values of  $\psi(-1,\eta_2)$  are 
		\begin{equation*}
			\psi(-1,\eta_2) _{\rm max} =\psi(-1,-1)  =2, \quad 
			\psi(-1,\eta_2) _{\rm min} =\psi(-1,1)  = 1. 
		\end{equation*}
		\item  For $ (\eta_1,\eta_2)\in \partial \mathcal{D}_2$,  $\psi(\eta_1,\eta_2) =\psi(1,\eta_2) =\frac{1}{2}(1-\eta_2).$  For  $\eta_2 \in[-1,1/2]$, we have  
		\begin{equation*}
			\psi(1,\eta_2) _{\rm max} =\psi(1,-1)  =1, \quad 
			\psi(1,\eta_2) _{\rm min} =\psi(1,1/2)  = \frac{1}{4}. 
		\end{equation*}
		\item  For $ (\eta_1,\eta_2)\in \partial \mathcal{D}_3$, $\psi(\eta_1,\eta_2) =\psi( \eta_1, 1/2) =\frac{1}{4}(3-2\eta_1).$
		For  $\eta_1 \in[1/2,1]$, we have  
		\begin{equation*}
			\psi(\eta_1,1/2) _{\rm max} =\psi(1/2,1/2)  =\frac{1}{2}, \quad 
			\psi(\eta_1,1/2) _{\rm min} =\psi(1,1/2)  = \frac{1}{4}. 
		\end{equation*}
	\end{itemize}
	Based on the above discussions,  the extreme values of $\psi(\eta_1,\eta_2)$ over $\mathcal{D}$ are  
	\begin{equation*}
		\psi(\eta_1,\eta_2) _{\rm max} =\psi(-1,-1)  =2, \quad
		\psi(\eta_1,\eta_2) _{\rm min} =\psi(1,1/2)  =\psi(1/2,1) = \frac{1}{4}.
	\end{equation*}
	It follows that 
	\begin{equation*}
		\mu_{\rm loc}(S_J) =  \max_{\boldsymbol{\theta} \in T^{\rm H_3}} \{|1-\omega \widetilde{D}_h^{-1} \widetilde{L}_h|\}
		=\max \left\{ |1-2\omega|, |1-\omega/4| \right\}. 
	\end{equation*}
	To minimize $\mu_{\rm loc}(S_J)$ over $\omega$, it requires that 
	$
	|1-2\omega|= |1-\omega/4|, 
	$
	which gives that $\omega =\frac{2}{2+1/4}=\frac{8}{9}$. Furthermore, $\mu_{\rm opt}=1-\omega/4=\frac{7}{9}$.
\end{proof}

\begin{theorem}\label{thm:Jacobi-H4-opt-mu} 
	For  coarsening by four ($q=4$),     $\tau=\frac{a}{a_1} \in \left [\frac{2-\sqrt{2}}{4}, 2\right]$ for $\boldsymbol{\theta} \in T^{\rm H_4}$.    Moreover, the optimal smoothing factor of weighted Jacobi relaxation scheme for the scalar Laplacian is given by
	\eqn
		\min_{\omega} \max_{\boldsymbol{\theta} \in T^{\rm H_4}} \{|\lambda(1-\omega \widetilde{D}_h^{-1}\widetilde{L}_h)|\} 
		&=\min_{\omega} \max \left\{|1- (2-\sqrt{2}) \omega/4|, |1-2\omega | \right\} \\
		 &=\frac{6+\sqrt{2}}{10-\sqrt{2}}\approx 0.864,
	\een
	provided that $\omega =\frac{8}{10-\sqrt{2}}\approx 0.932$.
\end{theorem}

\begin{proof}
	From the proof of Theorem \ref{thm:Jacobi-H3-opt-mu}, we know that the maximum of $\widetilde{D}_h^{-1} \widetilde{L}_h$ is 2. Next, we compute the  minimum of $ \widetilde{D}_h^{-1}\widetilde{L}_h=\frac{1}{2}(2-\eta_1-\eta_2)$   at the boundary of 
	$\mathcal{D}^{(4)}$, see \eqref{eq:H4-range}. Since $\psi(\eta_1,\eta_2) =\psi(\eta_2,\eta_1) $,  we only need to consider the following three boundaries of $\mathcal{D}^{(4)}$:  
	\begin{equation*}
		\partial \mathcal{D}_1 =\{-1\} \times [-1,1], \, \partial \mathcal{D}_2=\{1\} \times \left[-1,\sqrt{2}/2\right], \,
		\partial \mathcal{D}_3 =\left [\sqrt{2}/2,1\right]\times \left\{\sqrt{2}/2\right\}. 
	\end{equation*} 
	Following the same proof of Theorem \ref{thm:Jacobi-H3-opt-mu}, we can obtain the minimum of $\psi(\eta_1,\eta_2)$ is
	\begin{equation*}
		\psi(\eta_1,\eta_2)_{\rm min}=\psi(1, \sqrt{2}/2)=\frac{2-\sqrt{2}}{4}.
	\end{equation*} 
	To minimize $\mu_{\rm loc}(S_J)$ over $\omega$, it requires that $\omega =\frac{2}{2+(2-\sqrt{2})/4}=\frac{8}{10-\sqrt{2}}$. Furthermore, there obviously holds
	$\mu_{\rm opt}= 2\omega -1=\frac{6+\sqrt{2}}{10-\sqrt{2}}\approx 0.864,$
	which completes the proof.
\end{proof}
\begin{theorem}\label{thm:CJR-optimal}
	Let $\gamma =\frac{h^2}{4\sqrt{\alpha} }$ and $\omega_0= \frac{2+\gamma^2}{4+\gamma^2}$. For the CJR scheme, we have the following conclusions:
	\begin{enumerate}
		\item[(i)] For standard coarsening ($q=2$),   the optimal $\omega$ and  smoothing  factor   are  
		\begin{itemize}
			
			\item if $\alpha <\frac{h^4}{96}$ (i.e. $\gamma >\sqrt{6}$),     $\omega=\omega_0$ and   
			\begin{equation}\label{eq:CJR-q2-opt-case1}
				\mu_{\rm opt, CJR} = \sqrt{\frac{\gamma^2}{(4+\gamma^2)(1+\gamma^2)}  }\leq \sqrt{\frac{3}{35}}\approx 0.293.
			\end{equation} 
			\item if $\alpha \geq \frac{h^4}{96}$ (i.e.  $\gamma \leq \sqrt{6}$), $\omega=4/5$ and  
			\begin{equation}\label{eq:CJR-q2-opt}
				0.293\approx \sqrt{\frac{3}{35}}  \leq \mu_{\rm opt, CJR}=\frac{1}{5} \sqrt{ \frac{9+\gamma^2}{1+\gamma^2} } <\frac{3}{5}.
			\end{equation}  
		\end{itemize}
		
		\item[(ii)] For  coarsening by three ($q=3$),    the optimal $\omega$ and  smoothing  factor  are   
		\begin{itemize}
			
			\item if $\alpha <\frac{h^4}{224}$ (i.e. $\gamma >\sqrt{14}$),  $\omega=\omega_0$ and 
			\begin{equation}\label{eq:CJR-q3-opt-case1}
				\mu_{\rm opt, CJR} = \sqrt{\frac{\gamma^2}{(4+\gamma^2)(1+\gamma^2)}  }\leq \sqrt{\frac{7}{135}}\approx 0.228.
			\end{equation}
			\item if $\alpha \geq \frac{h^4}{224}$ (i.e. $\gamma \leq \sqrt{14}$),  $\omega=8/9$ and 
			\begin{equation}\label{eq:CJR-q3-opt}
				0.228\approx \sqrt{\frac{7}{135}}  \leq \mu_{\rm opt, CJR}=\frac{1}{9} \sqrt{ \frac{49+\gamma^2}{1+\gamma^2} } <\frac{7}{9}\approx 0.778.
			\end{equation}  
		\end{itemize}
		
		\item[(iii)]  For  coarsening by four ($q=4$),   the optimal $\omega$  and   smoothing  factor  are  
		\begin{itemize}
			
			\item if $\alpha <\frac{h^4(2-\sqrt{2})}{32(6+\sqrt{2})}$ (i.e. $\gamma >\sqrt{ (12+2\sqrt{2})/(2-\sqrt{2})}\approx \sqrt{25.3}$),  $\omega=\omega_0$ and
			\begin{equation}\label{eq:CJR-q4-opt-case1}
				\mu_{\rm opt, CJR} = \sqrt{\frac{\gamma^2}{(4+\gamma^2)(1+\gamma^2)}  }\leq 0.181.
			\end{equation} 
			\item if $\alpha \geq \frac{h^4(2-\sqrt{2})}{32(6+\sqrt{2})}$ (i.e.   $\gamma \leq \sqrt{ (12+2\sqrt{2})/(2-\sqrt{2})}\approx \sqrt{25.3}$),   $\omega=\frac{8}{10-\sqrt{2}}$ and
			\begin{equation}\label{eq:CJR-q4-opt}
				0.181 \leq \mu_{\rm opt, CJR}
				=\sqrt{\frac{ (4+\gamma^2)\omega^2-(4+2\gamma^2)\omega+1+\gamma^2}{1+\gamma^2}}< 0.864.
			\end{equation}  
		\end{itemize}
	\end{enumerate}
\end{theorem}

Before presenting the proofs, we first give some comments on the above results. From \eqref{eq:CJR-q2-opt-case1}, \eqref{eq:CJR-q3-opt-case1} and \eqref{eq:CJR-q4-opt-case1}, we notice that when $\gamma$ is large, coarsening by two, three and four gives the same optimal smoothing factor, which is true for $q$-coarsening ($q>4$). This suggests us we can consider $q$-coarsening which reduces the multigrid levels quickly and uses less CPU time. When $\gamma$ is small (e.g., due to a very small $\alpha$), from \eqref{eq:CJR-q2-opt}, \eqref{eq:CJR-q3-opt} and \eqref{eq:CJR-q4-opt}, it suggests to use standard coarsening since it gives a smaller optimal smoothing factor. A smaller $\alpha$ (while fixing $h$) leads to faster convergence is counter-intuitive considering the system becomes more ill-conditioned.
\begin{proof}
	We first consider standard coarsening. For $\boldsymbol{\theta}\in T^{\rm H_2}$,   Lemma \ref{lemma:Jacobi-H2-opt-mu} gives $\tau_1=1/2, \tau_0=2$ and $\tau_0=4/5$.  From \eqref{eq:general-minmax-form}, we have 
	\begin{equation}
		\mu_{\rm opt, CJR}=\min_{\omega}\max_{\boldsymbol{\theta}\in T^{\rm H_2}}|\lambda(I-\omega \widetilde{B}_{J}^{-1}\widetilde{A}_h)|
		=\min_{\omega \in [4/5,\infty)} \sqrt{\Psi(\omega)},
	\end{equation}
	where
	\begin{equation}\label{eq:Psi-defi}
		\Psi(\omega) =\frac{ |1-2 \omega |^2+|(1-\omega)\gamma|^2}{1+\gamma^2}=\frac{ (4+\gamma^2)\omega^2-(4+2\gamma^2)\omega+1+\gamma^2}{1+\gamma^2}:=\frac{\phi(\omega)}{1+\gamma^2}.
	\end{equation} 
	Since $(4+2\gamma^2)^2-4(4+\gamma^2)(1+\gamma^2)=-4\gamma^2<0$ and the symmetric axis of $\phi(\omega)$ is 
	\begin{equation}\label{eq:w0-form}
		\omega_0= \frac{4+2\gamma^2}{2(4+\gamma^2)} =\frac{2+\gamma^2}{4+\gamma^2},
	\end{equation} 
	there are two situations to minimize $\Psi(\omega) $: 
	\begin{enumerate}
		\item[(1)] If $\omega_0>\frac{4}{5}$, that is, $\gamma >\sqrt{6}$ or $\alpha <\frac{h^4}{96}$, then 
		\begin{equation*}
			\Psi(\omega) _{\rm min} =\Psi(\omega_0) =\frac{\gamma^2}{(4+\gamma^2)(1+\gamma^2)}=\frac{1}{5+\gamma^2+\frac{4}{\gamma^2}}\leq \frac{1}{5+6+4/6}=\frac{3}{35},
		\end{equation*} 
		which means that  $\mu_{\rm opt, CJR} = \sqrt{\Psi(\omega_0) }\leq \sqrt{3/35}\approx 0.293$.

		\item[(2)] If $\omega_0\leq \frac{4}{5}$, that is, $\gamma \leq \sqrt{6}$ or $\alpha \geq \frac{h^4}{96}$, then 
		\begin{equation*}
			\Psi(\omega) _{\rm min} =\Psi(4/5)=\frac{9+\gamma^2}{25(1+\gamma^2)}.
		\end{equation*} 
		It follows that  
		\begin{equation*}
		 \sqrt{3/35}= \mu_{\rm opt, CJR}(\sqrt{6}) \leq \mu_{\rm opt, CJR}(\gamma) =\sqrt{\Psi(4/5)}\leq \mu_{\rm opt, CJR}(0) =3/5,
		\end{equation*} 
		which completes the proof for $q=2$.
	\end{enumerate}
	
	Next, we consider coarsening by three.  For $\boldsymbol{\theta}\in T^{\rm H_3}$, from Theorem \ref{thm:Jacobi-H3-opt-mu}, we have $\tau_1=1/4, \tau_2=2$ and $\tau_0=8/9$.   
	From \eqref{eq:general-minmax-form}, we have 
	\begin{equation}
		\mu_{\rm opt, CJR}=\min_{\omega}\max_{\boldsymbol{\theta}\in T^{\rm H_3}}|\lambda(I-\omega \widetilde{B}_{J}^{-1}\widetilde{A}_h)|
		=\min_{\omega \in [8/9,\infty)} \sqrt{\Psi(\omega)},
	\end{equation}
	where $\Psi$ is defined in \eqref{eq:Psi-defi}.  From \eqref{eq:w0-form},  there are two situations to minimize $\Psi(\omega) $: 
	\begin{enumerate}
		\item[(1)] If $\omega_0>\frac{8}{9}$, that is, $\gamma >\sqrt{14}$ or $\alpha <\frac{h^4}{224}$, then 
		\eqn
			\Psi(\omega) _{\rm min} =\Psi(\omega_0) =\frac{\gamma^2}{(4+\gamma^2)(1+\gamma^2)}=\frac{1}{5+\gamma^2+\frac{4}{\gamma^2}} 
			\leq \frac{1}{5+14+4/14}=
			\frac{7}{135},
		\een 
		which means that  $\mu_{\rm opt, CJR}=\sqrt{\Psi(\omega_0) } \leq \sqrt{7/135}\approx 0.228$.

		\item[(2)] If $\omega_0\leq \frac{8}{9}$, that is, $\gamma \leq \sqrt{14}$ or $\alpha \geq \frac{h^4}{224}$, then 
		\begin{equation*}
			\Psi(\omega) _{\rm min} =\Psi(8/9)=\frac{49+\gamma^2}{81(1+\gamma^2)}.
		\end{equation*} 
		It follows that    
		\begin{equation*}
			 \sqrt{7/135}= \mu_{\rm opt, CJR}(\sqrt{14}) =\sqrt{\Psi(8/9)} \leq \mu_{\rm opt, CJR}(\gamma) 
			\leq \mu_{\rm opt, CJR}(0) =7/9,
		\end{equation*} 
		which completes the proof for $q=3$.
	\end{enumerate}

	Next, we consider coarsening by four.  For $\boldsymbol{\theta}\in T^{\rm H_4}$, from Theorem \ref{thm:Jacobi-H4-opt-mu}, we have $\tau_1=\frac{2-\sqrt{2}}{4}, \tau_2=2$ and $\tau_0=\frac{8}{10-\sqrt{2}}$.   
	From \eqref{eq:general-minmax-form}, we have 
	\begin{equation}
		\mu_{\rm opt, CJR}=\min_{\omega}\max_{\boldsymbol{\theta}\in T^{\rm H_3}}|\lambda(I-\omega \widetilde{B}_{J}^{-1}\widetilde{A}_h)|
		=\min_{\omega \in \left[8/(10-\sqrt{2}),\infty\right)} \sqrt{\Psi(\omega)},
	\end{equation}
	where $\Psi$ is defined in \eqref{eq:Psi-defi}.  From \eqref{eq:w0-form},  there are two situations to minimize $\Psi(\omega) $:
	\begin{enumerate}
		\item[(1)] If $\omega_0>\frac{8}{10-\sqrt{2}}$, that is, $\gamma >\sqrt{ (12+2\sqrt{2})/(2-\sqrt{2})}=\gamma_0$ or $\alpha <\frac{h^4(2-\sqrt{2})}{32(6+\sqrt{2})}$, then 
		\begin{align*}
			\Psi(\omega) _{\rm min} =\Psi(\omega_0) &=\frac{\gamma^2}{(4+\gamma^2)(1+\gamma^2)}=\frac{1}{5+\gamma^2+\frac{4}{\gamma^2}}
			 \leq \frac{1}{5+\gamma^2_0+\frac{4}{\gamma^2_0}}\\
			  &= \frac{1}{5+  (12+2\sqrt{2})/(2-\sqrt{2})+\frac{4(2-\sqrt{2})}{ 12+2\sqrt{2}}}\approx 0.0328,
		\end{align*} 
		which means that  $\mu_{\rm opt, CJR} =\sqrt{\Psi(\omega_0) } \leq \sqrt{0.0328}\approx 0.181.$
		\item[(2)] If $\omega_0\leq \frac{8}{10-\sqrt{2}}$, that is, $\gamma \leq \sqrt{ (12+2\sqrt{2})/(2-\sqrt{2})}$ or $\alpha \geq \frac{h^4(2-\sqrt{2})}{32(6+\sqrt{2})}$, then 
		\eqn
			\Psi(\omega) _{\rm min} &=\Psi(8/(10-\sqrt{2})) 
			 =\frac{ (4+\gamma^2)(8/(10-\sqrt{2}))^2-(4+2\gamma^2) (8/(10-\sqrt{2}))+1+\gamma^2}{1+\gamma^2}.
		\een
		It follows that  $\mu_{\rm opt, CJR}(\gamma) =\sqrt{\Psi(8/(10-\sqrt{2}))}$,
		and 
		\begin{equation*}
			 \mu_{\rm opt, CJR}( \gamma_0 ) \leq \mu_{\rm opt, CJR}(\gamma) 
			\leq \mu_{\rm opt, CJR}(0) =(6+\sqrt{2})/(10-\sqrt{2})\approx 0.864,
		\end{equation*} 
 		which completes the proof for $q=4$.  
	\end{enumerate}
\end{proof}
We remark that when $\alpha$ is very small compared to $h^4$,  the CJR scheme with the optimal damping parameter $\omega=\omega_0(\gamma)$ defined in \eqref{eq:w0-form}   converges very fast. However, when $\alpha$ is large, the optimal smoothing  factor gets too slow for practical use.  In the next subsection, we focus on a new BSR scheme, which can dramatically improve the convergence rates for the difficult cases with much larger $\alpha$ or smaller $h$.

\subsection{The mass-based BSR scheme}
In this subsection, we study the smoothing property of BSR scheme given in \eqref{eq:relaxation-scheme} with $B_h=B_m$.
Let $\widetilde{C}_h=b$, see \eqref{eq:exact-BSR-relaxation}. From \eqref{eq:mass-stencil} and \eqref{eq:symbol-form}, it follows from $Q_h=C_h^{-1}$  that
\begin{equation}\label{eq:symbol-Qh}
	\widetilde{Q}_h=\frac{h^2}{9}(4+2\cos\theta_1+2\cos\theta_2+\cos\theta_1\cos\theta_2)=1/b=\frac{1}{\widetilde{C}_h}.
\end{equation}
Thus, the symbol of   $B_m$ is 
\begin{equation*}
	\widetilde{B}_m = \begin{bmatrix}
		b& -1/\alpha     \\
		1 & a
	\end{bmatrix}.
\end{equation*} 
To find the eigenvalues of $\widetilde{B}_m^{-1}\widetilde{A}_h$, we compute the determinant of $\widetilde{A}_h-\lambda \widetilde{B}_m$:
\begin{equation*}
	| \widetilde{A}_h-\lambda\widetilde{B}_m| = 
	\begin{vmatrix}
		a - \lambda b      & -(1-\lambda)/\alpha   \\
		1 - \lambda  &   a-\lambda a
	\end{vmatrix} 
	=  (1-\lambda) \left( a(a-\lambda b) + (1-\lambda)/\alpha\right). 
\end{equation*}
Thus, the two real eigenvalues of  $\widetilde{B}_m^{-1}\widetilde{A}_h$  are $\lambda_1=1$ and 
\begin{equation}\label{eq:lambda2-expression}
	\lambda_2= \frac{1+\alpha a^2}{1+\alpha ab}. 
\end{equation}
Now, we have to find the range of $\lambda_2$ for $\boldsymbol{\theta} \in T^{\rm H_q}$. Before doing that, we first investigate  the optimal smoothing factor of  mass-based relaxation scheme for the  scalar Laplacian operator, that is, $S_h=I-\omega Q_h L_h$.
\begin{lemma}\label{lem:mass-optimal-mu} \cite{He21MassStokes}
	For standard coarsening ($q=2$),     $\frac{a}{b} \in \left [\frac{8}{9}, \frac{16}{9}\right]$ for $\boldsymbol{\theta} \in T^{\rm H_2}$.  Moreover, the optimal smoothing factor of mass-based relaxation scheme for the scalar Laplacian is  
	\begin{equation}
		\mu_{\rm opt}		=\min_{\omega}  \max \left\{|1-8/9 \omega|, |1-16/9\omega|\right\} =\frac{1}{3}\approx 0.333,
	\end{equation} 
	provided that $\omega =\frac{3}{4}$.
\end{lemma}
\begin{lemma}\label{lem:mass-optimal-mu-three} \cite{he2022optimal}
	For   coarsening by three  ($q=3$),   $\frac{a}{b} \in \left [\frac{5}{6}, \frac{16}{9}\right]$ for  $\boldsymbol{\theta} \in T^{\rm H_3}$. Moreover, the optimal smoothing factor of mass-based relaxation scheme for the scalar Laplacian is  
	\begin{equation}
		\mu_{\rm opt}		=\min_{\omega} \max  \left\{|1-5/6 \omega|, |1-16/9 \omega| \right\} =\frac{17}{47}\approx 0.362,
	\end{equation} 
	provided that $\omega =\frac{36}{47}$.
\end{lemma}

\begin{theorem}\label{thm:mass-optimal-mu-four}  
	For   coarsening by four  ($q=4$),    $\frac{a}{b} \in \left [\frac{3 -\sqrt{2}}{3}, \frac{16}{9}\right]$ for  $\boldsymbol{\theta} \in T^{\rm H_4}$. Moreover, the optimal smoothing factor of mass-based relaxation scheme for the scalar Laplacian is  
      \begin{equation}
		\mu_{\rm opt}
		 =\min_{\omega} \max  \left\{|1-(3 -\sqrt{2})/3 \omega|, |1-16/9 \omega| \right\} \\
		  =\frac{7+3\sqrt{2}}{25-3\sqrt{2}}\approx 0.542,
      \end{equation} 
	provided that $\omega =\frac{18}{25-3\sqrt{2}}\approx 0.867$.
\end{theorem}

\begin{proof}
	From \eqref{eq:symbol-Lh} and \eqref{eq:symbol-Qh}, we have 
	\begin{equation}\label{eq:symbol-QL}
		\widetilde{Q}_h\widetilde{L}_h=\frac{2}{9}(2-\cos\theta_1-\cos\theta_2)(4+2\cos\theta_1+2\cos\theta_2+\cos\theta_1\cos\theta_2).
	\end{equation} 
	Our goal is to find  the maximum and minimum of \eqref{eq:symbol-QL} over high frequencies. To do this,  let $\eta_1=\cos \theta_1$, $\eta_2=\cos \theta_2$.
	Then,  we rewrite  \eqref{eq:symbol-QL}  as $\widetilde{Q}_h \widetilde{L}_h=\frac{2}{9}\Upsilon(\eta_1,\eta_2)$, where 
	\begin{equation*} 
		\Upsilon(\eta_1,\eta_2) = (2-\eta_1-\eta_2)(4+2\eta_1+2\eta_2+\eta_1\eta_2).
	\end{equation*} 
	To find  the extreme values of $\Upsilon(\eta_1,\eta_2) $ over  $\mathcal{D}^{(4)}$, we start by computing the partial derivatives:
	\begin{equation*}
		\Upsilon'_{\eta_1}  =  -(2+\eta_2)(2\eta_1+\eta_2),\quad \Upsilon'_{\eta_2} =  -(2+\eta_1)(2\eta_2+\eta_1).
	\end{equation*}
	Let $\Upsilon'_{\eta_1}=\Upsilon'_{\eta_2}=0$ with $\boldsymbol{\theta} \in T^{\rm H_4}$. We have $\eta_1=\eta_2=0$. So   $\Upsilon(0,0)=8$ might be an extreme value.
	
	Next, we compute the extreme values of $\Upsilon(\eta_1,\eta_2)$ at the boundary of $\mathcal{D}^{(4)}$, see \eqref{eq:H4-range}. Since $\Upsilon(\eta_1,\eta_2) =\Upsilon(\eta_2,\eta_1) $,  we only need to consider the following three boundaries:  
	\begin{equation*}
		\partial \mathcal{D}_1 =\{-1\} \times [-1,1], \, \partial \mathcal{D}_2=\{1\} \times \left[-1,\sqrt{2}/2\right], \,
		\partial \mathcal{D}_3 =\left [\sqrt{2}/2,1\right]\times \left\{\sqrt{2}/2\right\}. 
	\end{equation*} 
	
	\begin{enumerate}
		\item For $ (\eta_1,\eta_2)\in \partial \mathcal{D}_1$,  we have  $\Upsilon(\eta_2,\eta_2) =\Upsilon(-1,\eta_2) = (3-\eta_2)(2+\eta_2)$.
		It is easy to see that the extreme values of  $\Upsilon(-1,\eta_2) $  over $\eta_2 \in[-1,1]$ are 
		\begin{equation*}
			\Upsilon(-1,\eta_2) _{\rm max}  = \Upsilon(-1,1/2)  = \frac{25}{4},\quad 
			\Upsilon(-1,\eta_2) _{\rm min}  = \Upsilon(-1,-1)  =4.
		\end{equation*}

		\item For $ (\eta_1,\eta_2)\in \partial \mathcal{D}_2$,  we have $\Upsilon(\eta_1,\eta_2) =\Upsilon(1,\eta_2) = 3(1-\eta_2)(2+\eta_2).$
		It follows that the extreme values of  $\Upsilon(1,\eta_2) $  over $\eta_2 \in[-1,\sqrt{2}/2]$ are given by
		\begin{equation*}
			\Upsilon(1,\eta_2) _{\rm max}  = \Upsilon(1,-1/2)  = \frac{27}{4},\quad 
			\Upsilon(1,\eta_2) _{\rm min}  = \Upsilon(1,\sqrt{2}/2)  =\frac{9-3\sqrt{2}}{2}.
		\end{equation*} 
		\item For $ (\eta_1,\eta_2)\in \partial \mathcal{D}_3$,  $\Upsilon(\eta_1,\eta_2) =\Upsilon(\eta_1,\sqrt{2}/2) =  (2-\sqrt{2}/2-\eta_1)(4+\sqrt{2}+2 \eta_1+\sqrt{2}\eta_1/2).$ Since  $\eta_1 \in[\sqrt{2}/2,1]$, we have 
		\eqn
			\Upsilon(\eta_1,\sqrt{2}/2) _{\rm max}  = \Upsilon(\sqrt{2}/2,\sqrt{2}/2)  = \frac{10-\sqrt{2}}{2},  \\
			\Upsilon(\eta_1,\sqrt{2}/2) _{\rm min}  = \Upsilon(1,\sqrt{2}/2)  =\frac{9-3\sqrt{2}}{2}.
		\een
	\end{enumerate}
	W conclude that  for  $\boldsymbol{\theta}\in T^{\rm H_4}$, $\Upsilon(\eta_1,\eta_2)_{\rm max} =\Upsilon(0,0)=8$ and $\Upsilon(\eta_1,\eta_2)_{\rm min}=\Upsilon(1,\sqrt{2}/2)=\frac{9-3\sqrt{2}}{2}$. It follows that when  $\boldsymbol{\theta}\in T^{\rm H_4}$,
	\begin{equation*} 
		(\widetilde{Q}_h \widetilde{L}_h)_{\rm max} =\frac{2}{9}\times 8=\frac{16}{9}, \quad   (\widetilde{Q}_h \widetilde{L}_h)_{\rm min} =\frac{2}{9}\times\frac{9-3\sqrt{2}}{2}=\frac{3-\sqrt{2}}{3}.
	\end{equation*} 
	Thus,
	\begin{equation*}
		\max_{\boldsymbol{\theta} \in T^{\rm {H_4}}} |1-\omega \widetilde{Q}_h \widetilde{L}_h|=\max \left\{\left|1- 16\omega/9\right|, \left|1- (3-\sqrt{2})\omega/3\right| \right\}.
	\end{equation*} 
	To minimize $\max_{\boldsymbol{\theta} \in T^{\rm {H_4}}} |1-\omega \widetilde{Q}_h \widetilde{L}_h|$ over $\omega$, it requires that
	$$\left|1- 16\omega/9 \right| =\left|1- (3-\sqrt{2})\omega/3\right|,$$ which gives $\omega=\frac{2}{16/9+ (3-\sqrt{2})/3}=\frac{ 18}{25-3\sqrt{2}}\approx 0.867$. Furthermore, 
	\begin{equation*}
		\mu_{\rm opt}=\left |1-\frac{16}{9}\times \frac{ 18}{25-3\sqrt{2}}\right |=\frac{7+3\sqrt{2}}{25-3\sqrt{2}}\approx 0.542,
	\end{equation*}
	which is the desired result.
\end{proof}

Since $\lambda_2$ is related to $a$ and $b$, from Lemmas \ref{lem:mass-optimal-mu}, \ref{lem:mass-optimal-mu-three}, and Theorem \ref{thm:mass-optimal-mu-four}, we are able to find the lower and upper bounds  for $\lambda_2$, which  gives a upper bound on the optimal smoothing factor for the mass-based BSR scheme. Such bounds are useful to predict the actual convergence factors.
\begin{theorem}\label{thm: BSR-upper-bound}
	For the  BSR scheme \eqref{eq:relaxation-scheme} with damping parameters $\omega=3/4$, $36/47$,$18/(25-3\sqrt{2})$ for $q=2,3,4$, respectively, the optimal smoothing factor with standard coarsening  is smaller  than $\frac{1}{3}\approx 0.333$,  with coarsening by three is smaller than $\frac{17}{47}\approx 0.362$, and with  coarsening by four  is smaller  than $\frac{7+3\sqrt{2}}{25-3\sqrt{2}}\approx 0.542$.
\end{theorem}
\begin{proof}
	We first compute
	\begin{equation}
		\lambda_2-\sigma =\frac{1-\sigma+\alpha a(a-\sigma b)}{1+\alpha ab}.
	\end{equation} 
	From Lemma \ref{lem:mass-optimal-mu},  for $\boldsymbol{\theta} \in T^{\rm H_2}$,  $\frac{a}{b} \in \left [\frac{8}{9}, \frac{16}{9}\right]$ and $a>0,b>0$.  It is easy to verify that 
	\begin{eqnarray*}
		\lambda_2-\frac{8}{9} &=&\frac{1-8/9 + \alpha a(a-8 b/9)}{1+\alpha ab}>0,\\
		\lambda_2-\frac{16}{9} &=&\frac{1-16/9 + \alpha a(a-16 b/9)}{1+\alpha ab}<0. 
	\end{eqnarray*}
	It follows that $\lambda_2 \in \left (\frac{8}{9}, \frac{16}{9}\right)$. Since $\lambda_1=1$, using Lemma \ref{lem:mass-optimal-mu}, we know that the optimal smoothing  factor 
	for  BSR scheme is smaller than $\frac{1}{3}$. Similarly, using Lemma \ref{lem:mass-optimal-mu-three} and Theorem \ref{thm:mass-optimal-mu-four}, we can obtain the desired results for coarsening by three and four.
\end{proof}
Since  it is very difficult to identify the extreme values of $\lambda_2$, we do not further explore the optimal smoothing factor of BSR.
We simple choose the constant damping parameters based on Lemmas \ref{lem:mass-optimal-mu} and \ref{lem:mass-optimal-mu-three}, and Theorem \ref{thm:mass-optimal-mu-four}, for $q=2, 3, 4$, respectively, for the corresponding BSR scheme for our control problems.
Although with non-optimal relaxation parameters, Theorem \ref{thm: BSR-upper-bound} tells us that we can obtain  upper bounds on the optimal smoothing factors for BSR scheme, which are much smaller than that of the CJR scheme when $\gamma$ is small, corresponding to the second situations in each coarsening in Theorem \ref{thm:CJR-optimal}. 
The actual convergence rates of BSR scheme can be far smaller than the given upper bounds, as reported in Table \ref{Ex1_BSR_factors}.

We comment that for the control constrained case, with some zero entries in the $(1,2)$-block of $A_h$, which in some sense can be treated as the limit case with $\alpha$ goes to infinity. The above smoothing analysis we derived for CJR and BSR schemes apply to any $\alpha>0$, so this partially explains why they work well for all Jacobian systems with varying 0-1 entries. Especially for BSR, it is unconditionally convergent, and the chosen constant damping parameters are independent of $\alpha$ and $h$ for $q=2,3,4$. The LFA of the CJR scheme indicates that the optimal damping parameters for BSR scheme may also depend on $\alpha$ and $h$, which remains an open problem.  We point out there are some studies of  LFA applied to PDEs with jumping and random coefficients \cite{kumar2019local,bolten2018fourier}, which might be helpful for  studying more complicated control problems.

\section{Numerical examples} \label{SecNum}
In this section, we present some numerical examples (with a unit square domain $\Omega=(0,1)^2$) to illustrate the effectiveness of our proposed multigrid algorithms. All simulations are  implemented with MATLAB on {a Dell Precision 5820 Workstation with Intel(R) Core(TM) i9-10900X CPU@3.70GHz and 64GB RAM},
where the CPU times (in seconds) are estimated by the timing functions \texttt{tic/toc}.
In our multigrid algorithms, we use the coarse operator from re-discretization with a coarse mesh step size $H=q h$, full weighting restriction and linear interpolation operators, W or V cycle with $\nu$-pre and no post smoothing iteration, the coarsest mesh step size $h_0\ge 1/8$ (depending on $q$), and the stopping tolerance $tol=10^{-10}$ based on reduction in relative residual norms. 
We will only use $\nu=1$ in our numerical tests unless otherwise specified.
The multigrid convergence factor of $k$-th iteration  is computed according to \cite{trottenberg2000multigrid}
\eq
\varrho^{(k)}=\left( \|r_k\|_2/\|r_0\|_2 \right)^{1/k},
\ee
where $r_k=b_h-A_h v_h^{(k)}$ denotes the residual vector after the $k$-th multigrid iteration. The initial guess $v_h^{(0)}$ is chosen as uniformly distributed random numbers in $(0,1)$.
We will record  $\varrho_W=\varrho^{(k)}$ for W-cycle, and $\varrho_V=\varrho^{(k)}$ for V-cycle as the measured multigrid convergence factors, where $k$ is the smallest integer such that the tolerance is achieved. Suppose the multigrid convergence factors for the CJR and BSR schemes are $\varrho_{\star,J}$ and $\varrho_{\star,S}$, where $\star=W, V$, respectively.  Then, for the same stopping tolerance the CJR scheme will need about $\eta:=\ln(\varrho_{\star, S})/\ln(\varrho_{\star, J})$ times multigrid iterations as the BSR scheme.
Based on our LFA predictions, we have $\eta\approx 2$ for $q=2$ and $\eta\approx 4$ for $q=3,4$, which implies that the inexact BSR scheme may take 2--4 times less CPU times than the CJR scheme, as indeed observed in Figure \ref{FigEx1}.

We compare the  CJR scheme with our derived optimal damping parameters
and our mass-based  BSR scheme with $\omega=3/4, 36/47,18/(25-3\sqrt{2})$ for $q=2,3,4$ respectively. 
To achieve an improved computational efficiency for the BSR scheme, we find that it is sufficient to inexactly solve the Schur complement system, see \eqref{eq:BSR-stage1},  in BSR scheme by 2 PCG iterations (with a diagonal preconditioner).  {In the following we refer to such a PCG-based inexact version of BSR as IBSR.}
In the control-constrained cases with $\beta\ge 0$, our multigrid solver is used to approximately solve the Jacobian linear systems within the SSN iterations, where the SSN stopping tolerance is also $tol=10^{-10}$ and the initial guess is chosen as the unconstrained solution with $\beta=0$.

We remark that both CJR and BSR schemes can be easily parallelized since they involve only matrix-vector multiplications, while the commonly used effective but very expensive collective Gauss-Seidel relaxation \cite{takacs2011convergence} is more difficult to parallelize (except in the special framework of red-black ordering) and also not easy to analyze. We did not compare with the collective Gauss-Seidel relaxation scheme since it costs far longer CPU times.
We highlight that the operation cost of each iteration in CJR scheme is cheaper than that of inexact BSR scheme, but the faster multigrid convergence factors of  BSR scheme seem to pay off in term of reduced overall CPU times. See Appendix A for the implementation detail of both CJR and BSR schemes (including IBSR).

\subsection{Example 1: $\beta=0$ without control constraints}
In this example, we choose $f$ and $g$ such that the exact solution of optimality system (\ref{opt1U}) read 
\[
y(\bm x)=\sin(2\pi x_1)\sin(2\pi x_2)\exp(x_1+x_2),\quad p(\bm x)=\sin(2\pi x_1)\sin(2\pi x_2)\exp(x_1-x_2).
\]

\begin{table}[hbt!]
	\centering   
	\caption{Measured multigrid convergence factors, $\varrho_W, \varrho_V$ vs. LFA predictions, $\mu^{\nu}_{\rm opt, CJR}$, for CJR scheme (with $\alpha=10^{-6}$). 
	}
	\begin{tabular}{|c|c||ccc|}\hline
		
		$q$,$N$&& $\nu$=1  &$\nu$=2 & $\nu$=3       \\   \hline
		\multirow{3}{*}{$2,256$}
		&$\mu^{\nu}_{\rm opt, CJR}$&  0.600 & 0.360 & 0.216\\ 
		&$\varrho_W$& 0.610&0.371&0.227 	 \\	
		&$\varrho_V$& 0.612&0.388&0.271 	 \\	
		\hline  
		\multirow{3}{*}{$3,243$}
		&$\mu^{\nu}_{\rm opt, CJR}$& 0.778& 0.605& 0.471\\  
		&$\varrho_W$&  0.785&0.617&0.485	 \\	
		&$\varrho_V$&  0.783&0.617&0.484	 \\	
		\hline
		\multirow{3}{*}{$4,256$}
		&$\mu^{\nu}_{\rm opt, CJR}$& 0.864& 0.747 & 0.645\\  
		&$\varrho_W$&  0.870&0.757&0.658 \\	
		&$\varrho_V$&  0.870&0.757&0.658 \\		  
		\hline
	\end{tabular} 
	\label{Ex1_CJR_factors}
\end{table}

\begin{table}[hbt!]
	\centering   
	\caption{Measured multigrid convergence factors, $\varrho_W, \varrho_V$ vs. LFA predictions, $\mu^{\nu}$, for exact and inexact BSR scheme (with $\alpha=10^{-6}$). 
	}
	\begin{tabular}{|c|c||ccc|cccc|}\hline
		&&\multicolumn{3}{c|}{Exact BSR}  &\multicolumn{4}{c|}{Inexact BSR (IBSR)}
		\\ 	\hline 	
		&&\multicolumn{3}{c|}{ }  &{1-PCG}&{\textbf{2-PCG}} &{3-PCG}& {4-PCG}
		\\ 	\hline 		 	
		$q$,$N$&& $\nu$=1  &$\nu$=2 & $\nu$=3 &  $\nu$=1  &$\nu$=1 & $\nu$=1 &  $\nu$=1    \\   \hline
		\multirow{3}{*}{$2,256$} 
		&$\mu^{\nu}$& 0.333&0.111&0.037&&&&\\
		&$\varrho_W$& 0.258&0.072&0.035&	0.430&	0.267&	0.265&	0.263	 \\	
		&$\varrho_V$& 0.258&0.092&0.050&	0.433&	0.274&	0.266&	0.263	 \\	
		\hline  
		\multirow{3}{*}{$3,243$}
		&$\mu^{\nu}$&0.362 &0.131&0.047&&&&\\
		&$\varrho_W$&0.284&0.127&0.074&	0.624&	0.345&	0.297&	0.285	 \\	
		&$\varrho_V$& 0.304&0.158&0.094&	0.628&	0.344&	0.318&	0.322	 \\	
		\hline
		\multirow{3}{*}{$4,256$}
		&$\mu^{\nu}$& 0.542&0.294&0.159&&&&\\
		&$\varrho_W$& 0.462&0.214&0.106&	0.734&	0.502&	0.479&	0.470	 \\	
		&$\varrho_V$& 0.462&0.225&0.105&	0.735&	0.503&	0.481&	0.474	 \\		  
		\hline
	\end{tabular}
	\label{Ex1_BSR_factors}
\end{table}

To verify the derived LFA smoothing factors indeed predict the practical multigrid convergence factors,
we report the measured multigrid convergence factors in Table \ref{Ex1_CJR_factors} for the CJR scheme and in Table \ref{Ex1_BSR_factors} for the BSR scheme, respectively. In Table \ref{Ex1_CJR_factors}, we compute the $\mu_{\rm opt, CJR}$ from \eqref{eq:CJR-q2-opt}, \eqref{eq:CJR-q3-opt}, and \eqref{eq:CJR-q4-opt}, and we observe the measured  convergence factors $\varrho_W, \varrho_V$ match with the LFA predictions, $\mu_{\rm opt, CJR}$. Notice for $\alpha=10^{-6}$ and $h=1/256$ or $h=1/243$ there holds $\gamma^2 \approx 0$, which leads to the second situations with larger smooth factors for CJR scheme in Theorem \ref{thm:CJR-optimal}. In Table \ref{Ex1_BSR_factors}, we compute $\mu^{\nu}$ with $\omega$ chosen in Theorem \ref{thm: BSR-upper-bound}. We find that these $\mu$ ($\nu=1$) equal to the upper bounds predicted in Theorem \ref{thm: BSR-upper-bound}, since $\frac{\alpha}{h^4}\approx 10^3$, and, then $\lambda_2= \frac{1+\alpha a^2}{1+\alpha ab}\approx \frac{\alpha a^2}{\alpha ab}=\frac{a}{b}$, see \eqref{eq:lambda2-expression}, which reduces to the mass approximation to the scalar Laplacian.  Although the measured W-cycle and V-cycle convergence factors are smaller than LFA predictions, it is not surprising since  LFA does not take account of the influence of boundary conditions. Table \ref{Ex1_BSR_factors} also includes the results of four inexact BSR schemes with the Schur complement system being solved by only 1--4 PCG iterations, which shows 2--3 PCG iterations are sufficient to achieve similar convergence factors as the expensive exact BSR scheme (compare the columns with $\nu=1$), and there is no benefit to use 4 or more PCG iterations. 
In other words, BSR only requires a cheap rough approximate solution to the Schur complement system to be an effective smoother that effectually dampens high-frequency errors.
Hence,  we will use the inexact BSR based on 2 PCG iterations for better computational efficiency.  
In fact,  using a few weighted Jacobi iterations as the inexact solver also gives similar numerical results and hence they are omitted for the same of brevity.

To verify our theoretical conclusions in Theorem \ref{thm:CJR-optimal}, in Figure \ref{FigEx1JAC} we compare the standard CJR relaxation based on a fixed damping parameter $\omega=4/5$ (denote as CJR-F) often used in other studies, for example, \cite{engel2011multigrid}, and our derived optimized  CJR relaxation based on a $\gamma$-dependent damping parameter (by Theorem \ref{thm:CJR-optimal}), which shows the  CJR relaxation indeed provides faster convergence rates whenever $\gamma^2=h^4/(16\alpha)\gg 6$ (e.g., with a  very small $\alpha$ or a large $h$).
However, as shown in the right plot of Figure \ref{FigEx1JAC}, there is no obvious convergence difference between CJR and CJR-F when $\gamma^2\approx 6$ in the finest level.
In particular, from the first situations in Theorem \ref{thm:CJR-optimal}, the multigrid convergence rate with the  CJR scheme would become even smaller if $\gamma^2$ gets larger (in fact, $\lim_{\gamma\rightarrow \infty }\mu_{\rm opt, CJR} = \sqrt{\frac{\gamma^2}{(4+\gamma^2)(1+\gamma^2)}}\rightarrow  0$), as shown in Figure \ref{FigEx1JAC}.
Hence, the optimized CJR relaxation works very effectively for a very small $\alpha$ while fixing $h$, which removes the undesirable convergence condition $\alpha\ge c h^4$ obtained in the existing multigrid algorithms \cite{Schberl2011,engel2011multigrid}.
Our LFA is crucial to this surprising result.
\begin{figure}[hbt!]
	\begin{center}
		\includegraphics[width=1\textwidth]{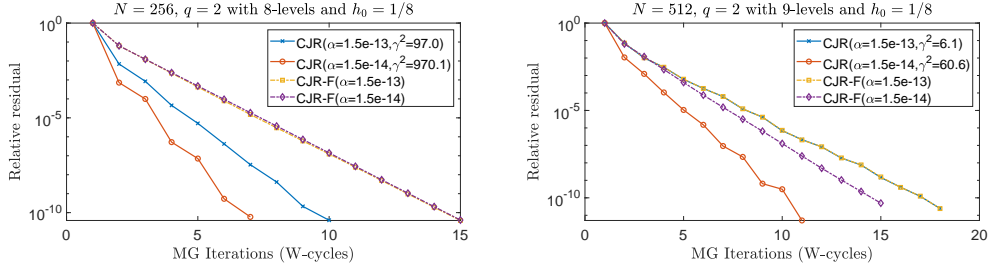}   
	\end{center} 
	\caption{The comparison of CJR with a fixed damping weight, $4/5$, (CJR-F) and our optimized CJR (depending on $\gamma$).}
	\label{FigEx1JAC}
\end{figure}

\begin{figure}[hbt!]
	\begin{center}
		\includegraphics[width=1\textwidth]{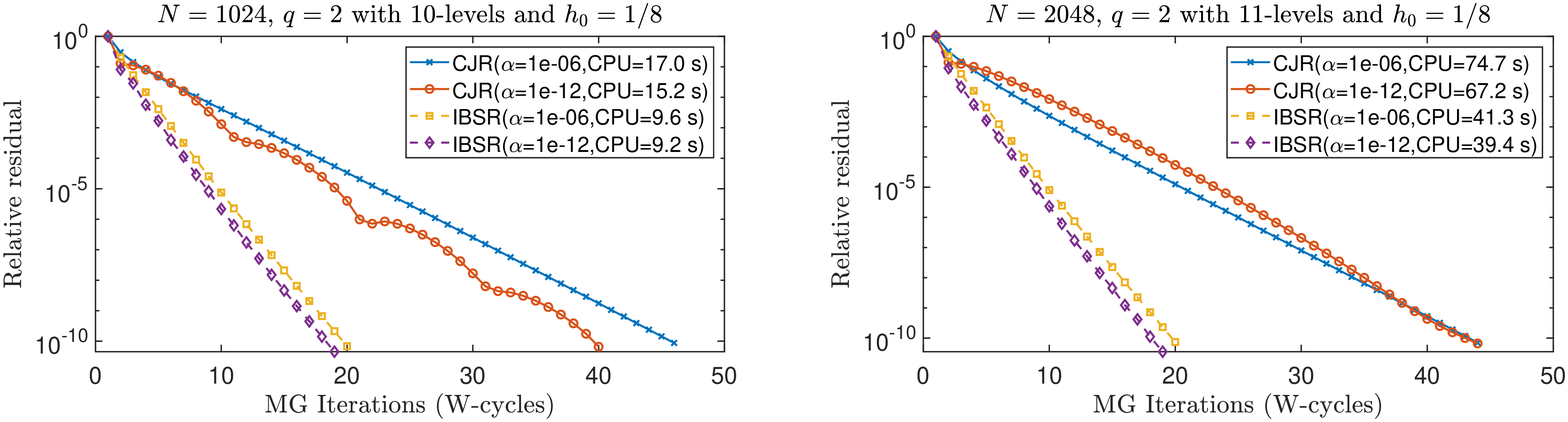}
		\includegraphics[width=1\textwidth]{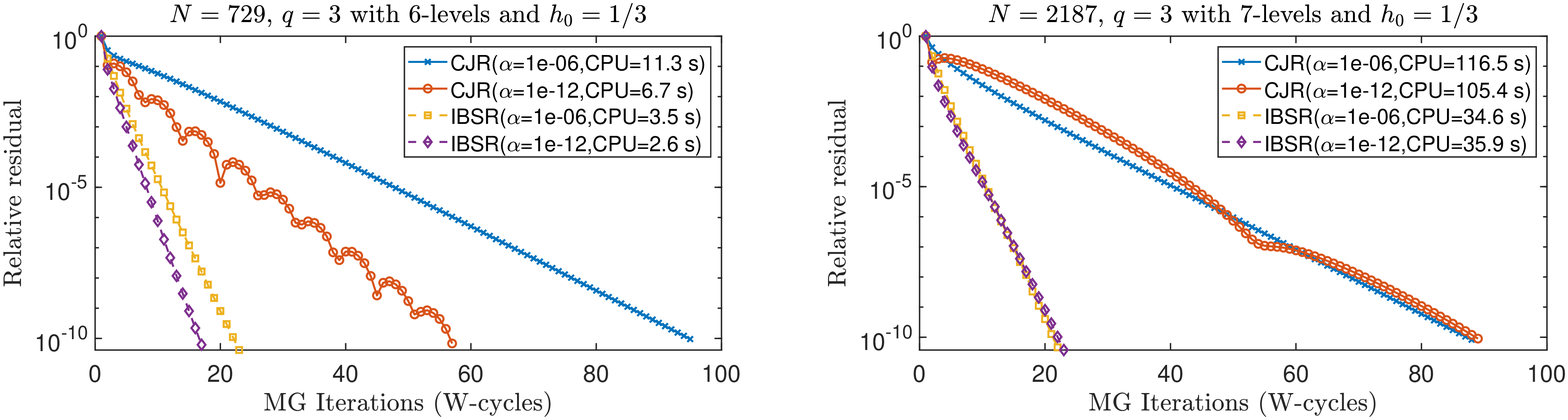}
		\includegraphics[width=1\textwidth]{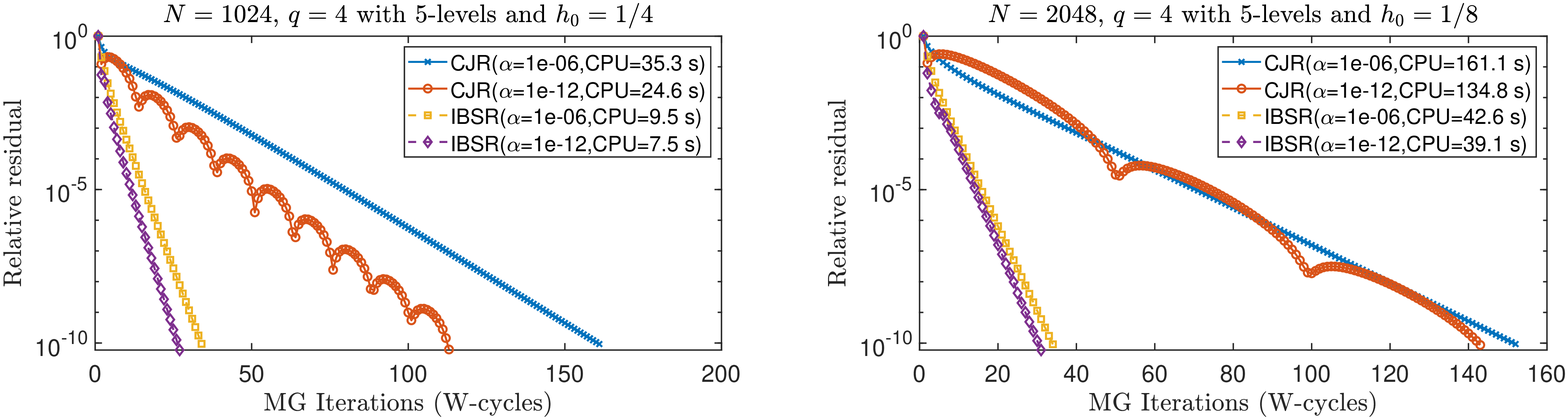}
	\end{center} 
	\caption{The comparison of multigrid iterations and CPU times of CJR and IBSR schemes for different $N$, $\alpha$ and $q$ (top row: $q=2$, middle row: $q=3$, bottom row: $q=4$). Notice the system sizes corresponding to $q=3$ are different from $q=2,4$ for suitable coarsening.}
	\label{FigEx1}
\end{figure}
Figure \ref{FigEx1} compares the multigrid iterations and CPU times for both CJR and BSR schemes with different mesh size $N$, regularization parameter $\alpha=10^{-6},10^{-12}$, and coarsen ratios $q=2,3,4$.
As anticipated, our proposed inexact BSR scheme with faster convergence factors takes much less number of multigrid iterations than the CJR scheme,
and it also achieves about 2--4 times speed up in CPU times.
In particular, our inexact BSR scheme leads to only mildly degrading convergence rates with respect to the coarsen ratios $q=2,3,4$, which is a significant advantage over the CJR scheme that shows greatly deteriorated convergence rates.
Notice that IBSR scheme for a very small $\alpha$ (compared with a fixed $h^4$) seems lead to faster convergence rates, which is expected since the corresponding Schur systems become more diagonally dominant. The observed interesting non-monotone  convergence of the CJR scheme with $\alpha=10^{-12}$ is likely due to our used $\gamma$-dependent relaxation parameters at different levels.
For the CJR scheme, there is very little or even no computational benefit to use the non-standard coarsen factors $q=3,4$ other than $q=2$ for the second situations in Theorem \ref{thm:CJR-optimal}, but the inexact BSR scheme seems to work reasonably well with the non-standard coarsen factors $q=3,4$.
Notice that a larger $q$ leads to less levels of coarse operators with much smaller dimensions, which implies the inexact BSR scheme is especially suitable for those large-scale problems that may exceed the memory limit if using standard coarsening.

\subsection{Example 2: $\beta>0$ with control constraints}
In this example, we consider the following data:
\[
f\equiv 0,\quad g(\bm x)=\sin(2\pi x_1)\sin(2\pi x_2)\exp(2 x_1)/6,\quad u_0=-30, \quad u_1=30.
\]

\begin{figure}[hbt!]
	\begin{center}
		\includegraphics[width=1\textwidth]{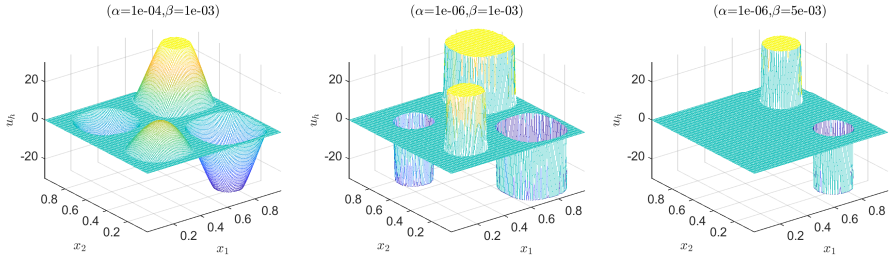} 
	\end{center} 
	\caption{Comparison of computed optimal control  with different values of $\alpha$ and $\beta$ ($N=128$).
		Notice that a larger $\beta$ leads to more sparsity,
		while a smaller $\alpha$ gives larger active set (i.e.,  the control constraints are active or attained).
	}
	\label{FigEx2control}
\end{figure}  
\begin{figure}[hbt!]
	\begin{center}
		\includegraphics[width=1\textwidth]{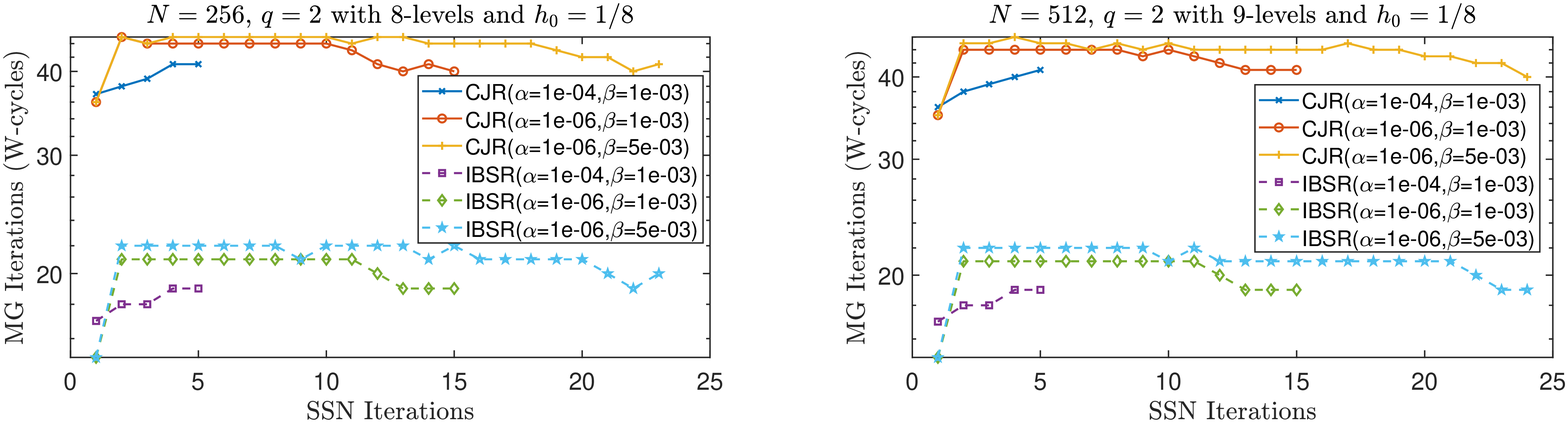} 
		\includegraphics[width=1\textwidth]{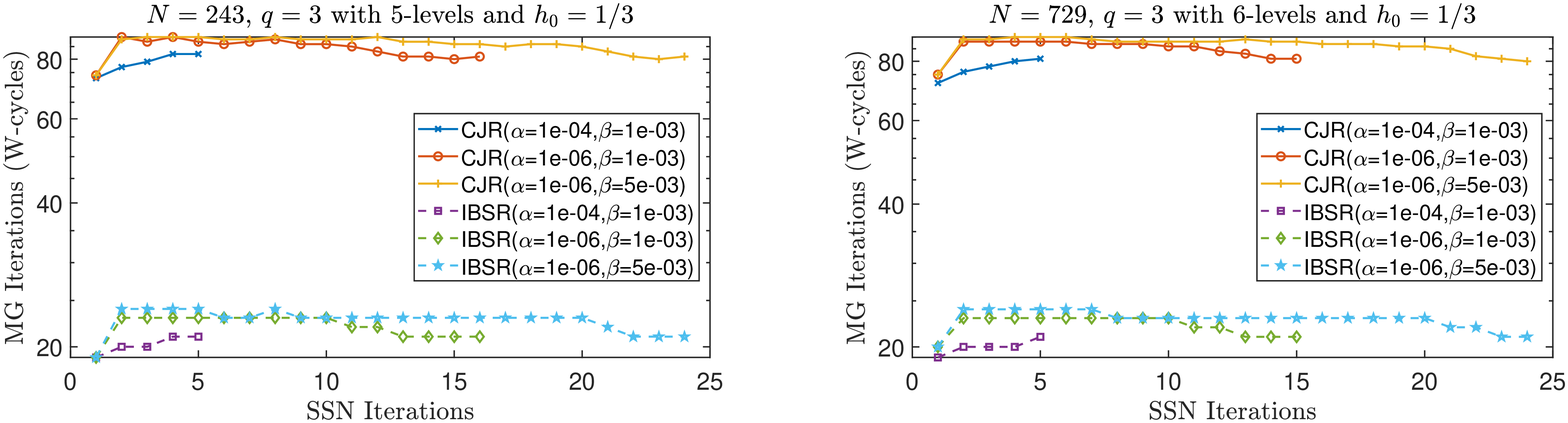} 
		\includegraphics[width=1\textwidth]{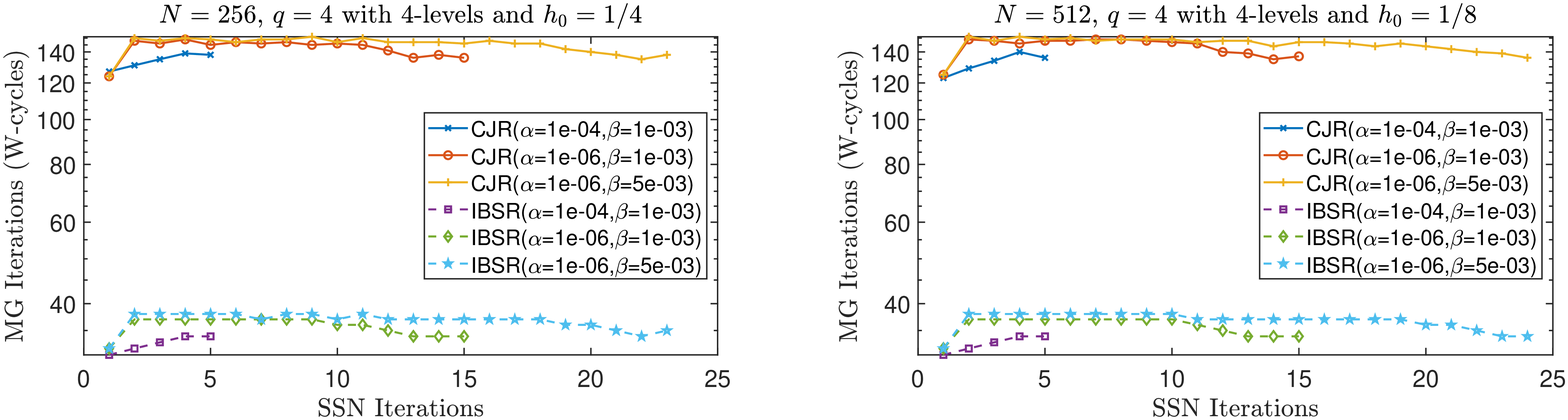} 
	\end{center} 
	\caption{The comparison of multigrid W-cycle iteration numbers  of CJR and IBSR schemes with different $N$, $\alpha,\beta$, and $q$ (top row: $q=2$, middle row: $q=3$, bottom row: $q=4$). Notice the system sizes corresponding to $q=3$ are different from $q=2,4$ for suitable coarsening.}
	\label{FigEx2W}
\end{figure}
Figure \ref{FigEx2control} illustrates how the values of parameters $\alpha$ and $\beta$ affect the shape and sparsity pattern of the computed optimal control, where a larger $\beta$ leads to more sparse control and a smaller $\alpha$ gives larger active set (i.e.,  the control constraints are active).  The knowledge of such sparsity pattern provides a priori
information on the most effective regions to place control devices \cite{Stadler2007}.

In Figure \ref{FigEx2W}, we compare the used number of multigrid iterations (W cycles) of CJR  and IBSR schemes at each SSN iteration, where the required number of SSN iterations seem to depend on the regularization parameters $\alpha$ and $\beta$.
Nevertheless, we highlight that the used multigrid iteration numbers for solving each Jacobian system are about the same as in the  unconstrained linear case with $\beta=0$ for both CJR and IBSR schemes (compare with the iteration number in Figure \ref{FigEx1} of Example 1),
which indicates our proposed multigrid solver indeed has a very robust convergence rate with respect to the relevant parameters $\alpha$ and $\beta$ that determines the 0-1 pattern of $G^{(k)}_h$ in the (1,2) block of Jacobian matrix.  Again, we observe that  the proposed IBSR scheme takes much less iteration numbers than the CJR scheme for $q=2, 3, 4$, to achieve the tolerance. Hence IBSR is recommended in practice.

\section{Conclusions} \label{SecFinal}
In this paper, we first performed LFA with coarsening by two, three, and four  for the widely used CJR scheme for elliptic optimal control problems to obtain the optimal damping parameter and optimal smoothing factor that were not available in literature.
As an improvement of the CJR scheme, we then proposed and analyzed a mass-based BSR scheme, which was shown to deliver faster convergence rates than the CJR scheme. Numerical results from both unconstrained and constrained examples confirmed the parameter-robust convergence rates of both CJR and BSR schemes.
The inexact BSR scheme outperforms the CJR scheme in the practical situations with $\alpha\ge ch^4$ for some constant $c$.  Moreover, Coarsening by three or four with BSR is competitive with coarsening by two. 
It is possible to generalize our approach to the  more difficult control problems with state constraints \cite{Borzi2008,Gong2010,Pearson2012,Casas2014,Brenner2017,Brenner2018}.
The proposed standalone solvers can be used as an effective preconditioner for Krylov subspace methods, which will be investigated in the future.

\appendix

\section{Practical implementation of the CJR and BSR schemes}
In this appendix, we briefly describe the implementation of both CJR and BSR schemes, with the goal of providing a overview of each scheme's main operation costs.

Recall the CJR and BSR scheme reads
\begin{equation}\label{eq:relaxation-scheme2}
	v_h^{k+1}=v_h^{k}+\omega B^{-1}_h(b_h-A_hv_h^k),
\end{equation}
where $B_h=\begin{bmatrix}
	D_h& -  I/\alpha\\
	I & D_h
\end{bmatrix}$ with $D_h={\rm diag}(L_h)$ for CJR and $B_h=\begin{bmatrix}
	C_h& -  I/\alpha\\
	I & L_h
\end{bmatrix}$ with $C_h=Q_h^{-1}$ for BSR.

Let $\bmt r_f\\r_g \emt=b_h-A_hv_h^k$ and $\bmt w_f\\w_g \emt=B_h^{-1}\bmt r_f\\r_g \emt$.
Then each relaxation iteration needs to solve  
\[
\bmt
D_h& -  I_h/\alpha\\
I_h & D_h
\emt 
\bmt w_f\\w_g \emt=\bmt r_f\\r_g \emt \qquad\text{and}\qquad
\bmt C_h& -  I_h/\alpha\\
I_h & L_h\emt 
\bmt w_f\\w_g \emt=\bmt r_f\\r_g \emt
\]
for CJR and BSR, respectively.
Such block systems can be solved by block Gaussian elimination.

For CJR scheme, the update can be computed by the following two steps:
\begin{eqnarray}
	w_g&=& (D_h+ D^{-1}_h/\alpha)^{-1}\left(r_g- D^{-1}_h r_f\right), \label{eq:CJR-stage1}\\
	w_f&=& D^{-1}_h (r_f+ w_g/\alpha),\nonumber
\end{eqnarray}
where $D_h^{-1}$ and $(D_h+ D^{-1}_h/\alpha)^{-1}$ are fast to compute since $D_h$ is diagonal.

For BSR scheme, by $C^{-1}_h=Q_h$, the update  can be computed in two steps:  
\begin{eqnarray}
	w_g&=&(L_h+ Q_h/\alpha)^{-1}\left(r_g- Q_h r_f\right), \label{eq:BSR-stage1}\\
	w_f&=& Q_h (r_f+ w_g/\alpha),\nonumber
\end{eqnarray} 
where the matrix $(L_h+  Q_h/\alpha)$ is expensive to invert exactly.  
In several applications \cite{he2021novel}, it has been shown that an inexact solve of the Schur complement system \eqref{eq:BSR-stage1} is sufficient to obtain similar  convergence rates as the expensive exact solve. There are many  fast iterative algorithms, such as weighted Jacobi method, multigrid V-cycles, and PCG method, to approximately solve the symmetric and positive definite (SPD) Schur system \eqref{eq:BSR-stage1}. 
Based on the results in Example 1 (see Table \ref{Ex1_BSR_factors}), we suggest to inexactly solve the Schur system \eqref{eq:BSR-stage1} by using 2 PCG iterations with the   diagonal preconditioner $\textrm{diag}(L_h+  Q_h/\alpha)$. 
It is desirable to develop more efficient inexact solvers for the Schur system \eqref{eq:BSR-stage1} based on its special structure for the purpose of effective smoothing.
Obviously, the above described CJR and BSR schemes with minimal modification can also be applied to handle the Jacobian system (\ref{opt1Jh}).
%

\bibliographystyle{siamplain}
\bibliography{Control_Laplacebib,EllipticControl,../LFA,../waveControl,../waveControl2}
\end{document}